\documentclass[12pt,leqno]{amsart}
\usepackage{amsmath, amssymb, amscd, amsfonts,    stmaryrd, turnstile, mathrsfs, eucal,color}

\definecolor{dullmagenta}{rgb}{0.4,0,0.4}

\definecolor{darkblue}{rgb}{0,0,0.4}

\usepackage[colorlinks=true, pdfstartview=FitV, linkcolor=red, citecolor=blue, urlcolor=darkblue]
{hyperref}

\newtheorem{theorem}{Theorem}[section]
\newtheorem{lemma}[theorem]{Lemma}
\newtheorem{proposition}[theorem]{Proposition}
\newtheorem{corollary}[theorem]{Corollary}

\theoremstyle{definition}
\newtheorem{definition}[theorem]{Definition}
\newtheorem{example}[theorem]{Example}
\newtheorem{remark}[theorem]{Remark}

\begin{document}

\title[Quasi-maximal ideals and ring extensions ]{Quasi-maximal ideals and ring extensions}

\author[G. Picavet and M. Picavet]{Gabriel Picavet* and Martine Picavet-L'Hermitte}
\address{Math\'ematiques \\
63670, Le Cendre\\
 France}
\email{picavet.gm (at) wanadoo.fr}

\begin{abstract} Alan and al. defined and studied quasi-maximal ideals. We add a comprehensive characterization of these ideals, introducing submaximal ideals. The conductor of a finite minimal extension $R\subset S$ is quasi-maximal in $S$. This allows us to give a new characterization of these extensions. We also examine the links between quasi-maximal ideals and Badawi 2-absorbing ideals.
\end{abstract}


\keywords {minimal extension, conductor of an extension, quasi-maximal, submaximal, 2-absorbing, primary, primal ideals, integral extension, covering ideal, support of a module.\\ *Corresponding author}

\maketitle

\section{Introduction and Notation}

The paper being rather lengthy, we will give a brief introduction. Any unexplained terminology can be found by the reader in the ad hoc section.  
We consider the category of commutative and unital rings. If $R\subseteq S$ is a (ring) extension, $([R,S],\subseteq)$ denotes the lattice
 of all $R$-subalgebras of $S$.

In \cite[Definition 1]{AKKT}, the notion of quasi-maximal ideals was   introduced: a proper ideal $I$ of a ring $R$ (i.e. distinct from $R$) is said to be a quasi-maximal ideal if, for any $a\in R\setminus I$, either $I+Ra=R$ or $I+Ra$ is a maximal ideal of $R$. In particular, it was shown \cite[Proposition 3]{AKKT} that there are three types of quasi-maximal ideals. 

In Section 2, we give a lot of equivalences for an ideal to be quasi-maximal. A characterization of quasi-maximal ideals is given by Theorem \ref{2.6} which says that a non maximal ideal $I$ of a ring $R$ is quasi-maximal if and only if $I$ is strictly contained only in maximal ideals, 
 in other words, $I$ is submaximal.

Section 3 is devoted to the behavior of quasi-maximal ideals in extensions or through ring morphisms. Let $f:R\to S$ be a ring morphism. If $I$ is a quasi-maximal ideal of $R$ or $J$ a quasi-maximal ideal of $S$, we give conditions in order that $f(I)$ or $f^{-1}(J)$ be quasi-maximal ideals.

We consider in Section 4 minimal extensions and examine how they are linked to quasi-maximal ideals. This section is quite long because we consider the different types of quasi-maximal ideals and minimal extensions. In Theorem \ref{4.1} we get that if $R\subset S$ is a finite minimal extension, then the conductor of the extension is a quasi-maximal ideal of $S$ of the same type as the minimal extension $R\subset S$. This Theorem has a kind of converse with Theorem \ref{4.15}. 
 
In Section 5, we consider the link between quasi-maximal ideals and 2-absorbing ideals. Indeed, it is shown in \cite[Theorem 1]{AKKT} that a quasi-maximal ideal is a 2-absorbing ideal, a notion introduced by Badawi in \cite{Ba}. With Theorem \ref{5.2}, we get a converse for a  
 2-absorbing ideal to be quasi-maximal with additional assumptions.
  
2-absorbing ideals are primal ideals, and so quasi-maximal ideals are primal ideals. The last section of this paper is an Appendix that uses a saturated multiplicative set introduced by the first author in \cite{PICANAL} and  specially adapted to quasi-maximal ideals.

A {\it (semi)-local} ring is a ring with (finitely many) one maximal ideal(s). Now, some notation for a ring $R$: $\mathrm{Spec}(R)$ and $\mathrm {Max}(R)$ are the set of prime and maximal ideals of a ring $R$ and $\mathrm{t}(R)$ is the total quotient ring of $R$. We denote by $\mathrm {QMax}(R)$ the set of quasi-maximal ideals of $R$. For an extension $R\subseteq S$ and an ideal $I$ of $R$, we write $\mathrm{V}_S(I):=\{P\in\mathrm{Spec}(S)\mid I\subseteq P\}$ and $\mathrm{D}_S(I):=\mathrm {Spec}(S)\setminus\mathrm{V}_S(I)$. The support of an $R$-module $E$ is $\mathrm{Supp}_R(E):=\{P\in\mathrm{Spec}(R)\mid E_P\neq 0\}$, and we set $\mathrm{MSupp}_R(E): =\mathrm{Supp}_R(E)\cap\mathrm{Max}(R)$. We denote the length of $E$ by $\mathrm{L}_R(E)$. When $R\subseteq S$ is an extension, we set $\mathrm{Supp}(T/R):=\mathrm {Supp}_R(T/R)$ and $\mathrm{Supp}(S/T):=\mathrm{Supp}_R(S/T)$ for each $T\in[R,S]$, unless otherwise specified. The set of zero divisors of $ R$ is denoted by $\mathrm Z(R)$, $\mathrm{U}(R)$ is the set of units of $ R$, $\mathrm{Nil}(R)$ is the set of nilpotents elements of $R$ and $\mathrm J(R)$ is its Jacobson radical. For an ideal $ I$ of $R$, the radical of $I$ is denoted by $\sqrt[R]I$ (or $\sqrt I$ if there is no confusion), and if the set of proper ideals of $R$ containing $I$ (including $I$) is finite, we denote its cardinality by $\mathrm o(I)$. If $R$ has finitely many ideals, we denote by $\mathcal O(R)$ the cardinality of the set of proper ideals of $R$ (actually $\mathcal O(R)=\mathrm o(0)$). We recall that a ring $R$ is $\mathrm J$-{\it regular} if $R/\mathrm J(R)$ is regular. At last, $(R:S)$ is the conductor of an extension $R\subseteq S$. 

An extension $R\subseteq S$ is called an {\it i-extension} if the natural map $\mathrm{Spec}(S)\to\mathrm{Spec}(R)$ is injective. An integral extension $R\subseteq S$ is called {\it infra-integral} \cite{Pic 2}, (resp. {\it subintegral} \cite{S}) if all its residual extensions are isomorphisms (resp$.$; and is an {\it i-extension}).

Finally, $|X|$ is the cardinality of a set $X$, $\subset$ denotes proper inclusion and for a positive integer $n$, we set $\mathbb{N}_n:=\{1,\ldots,n \}$. If $R$ and $S$ are two isomorphic rings, we will write $R\cong S$. The characteristic of a ring $R$ is denoted by $\mathrm{c}(R)$.
  
  \section{Properties of quasi-maximal ideals}
  
Let $I$ and $J$ be two ideals of a ring $R$. Following \cite[D\'efinition, p.10-02]{GU}, we say that $J$ {\it covers} $I$, that we denote by $I\prec J$ (or $J\succ I$), if $I\subset J$ and there is no ideal properly between $I$ and $J$. This property is characterized by \cite[Corollaries 1 and 2, p.237]{ZS}. 

\begin{proposition}\label{2.7} For any $P\subset Q$ in $\mathrm{Spec}(R)$, then  $Q\nsucc P$. 
\end{proposition}
\begin{proof} If  there exists $P\subset Q$ in $\mathrm{Spec}(R)$ such that $Q\succ P$, by \cite[Corollary 1, p.237]{ZS}, there is  some $M\in\mathrm{Max}(R)$ such that $MQ\subseteq P$. This implies that either $ M\subseteq P$, so that $P=M$, a contradiction with $P\subset Q$, or $Q\subseteq P$ again a contradiction for the same reason.
\end{proof}

We recall the three types of quasi-maximal ideals:
\begin{proposition}\label{2.70} \cite[Proposition 3]{AKKT} A quasi-maximal ideal $I$ of a ring $R$ satisfies one of the following mutually exclusive properties  :
\begin{enumerate}
\item $I$ is a maximal ideal of $R$.
\item $I$ is the intersection of two (distinct) maximal ideals of $R$.
\item There exists a maximal ideal $M$ of $R$ such that $I$ is $M$-primary and satisfies $M^2\subseteq I\subset M$.
\end{enumerate}
\end{proposition}

Now the conductor of a minimal integral extension has exactly one of the three types of a quasi-maximal ideal, as we will see in Section 4. This suggested to us to give the next definition because \cite[Proposition 3]{AKKT} has a kind of converse. In order to know what type of quasi-maximal ideals we are dealing with, we will term them with the same name as the one of the corresponding minimal integral extension. See Section 4 for the link between these two notions.

\begin{definition}\label{2.2} An ideal $I$ of a ring $R$ is said to be 
\begin{enumerate}
\item {\it inert} if $I\in\mathrm{Max}(R)$.
\item {\it decomposed} if $I$ is the intersection of two (distinct) maximal ideals of $R$.
\item {\it ramified} if there exists $M\in\mathrm{Max}(R)$ such that $I$ is $ M$-primary, $M^2\subseteq I\subset M$ and $M=I+Ra$ for any $a\in M\setminus I$, so that $M=\sqrt I$.
\end{enumerate}
\end{definition}

In the proofs considering the different types, we often drop the inert case since they are obvious.

Theorem \ref{2.3} gives a lot of equivalent statements for an ideal to be quasi-maximal. It completes \cite[Proposition 3]{AKKT} by giving a converse and explains why Definition \ref{2.2} characterizes quasi-maximal ideals. More precisely, we write $I\in\mathrm{Q_iMax}(R)$ (resp. $\mathrm {Q_r Max}(R),\ \mathrm{Q_dMax}(R)$) when $I$ is inert, (resp. ramified, decomposed) in $R$. In fact, $\mathrm{Q_iMax}(R)=\mathrm{Max}(R)$. In Section 5, Theorem \ref{5.2} gives an additional equivalence using 2-absorbing ideals introduced by Badawi in \cite{Ba}. If $R\subseteq S$ is an extension, $I\in\mathrm{QMax(R)}$ and $J\in\mathrm{QMax(S)}$, we say that $I$ and $J$ are {\it homotypic} if they have the same type.

\begin{theorem}\label{2.3} The statements below are equivalent for a proper ideal $I$ of a ring $R$ and imply $\mathrm{V}_R(I)\subseteq\mathrm{Max}(R)$:
\begin{enumerate}
\item  $I\in\mathrm{QMax}(R)$.
\item Any proper ideal of $R/I$ is in $\mathrm{QMax}(R/I)$.
\item $0\in\mathrm{QMax}(R/I)$.
\item $1\leq\mathrm{L}_R(R/I)\leq 2$ and $|\mathrm{V}_{R/I}(0)|\leq 2$. 
\item $\mathcal O(R/I)\leq 3$ with $\mathcal O(R/I)=3$ if and only if $|\mathrm{Max}(R/I)|=2$.
\item $\mathrm o(I)\leq 3$ with $\mathrm o(I)=3$ if and only if $|\mathrm {V}_R(I)\cap \mathrm{Max}(R)|= 2$.
\item $I\in\mathrm{Q_iMax}(R)\cup\mathrm{Q_rMax}(R)\cup\mathrm {Q_dMax}( R)$. 
\end{enumerate}
\end{theorem}
\begin{proof} (1)$\Rightarrow$(2) by \cite[Corollary 2]{AKKT}.

(2)$\Rightarrow$(3): Obvious.

(3)$\Rightarrow$(4): We consider the three cases for $0\in\mathrm {QMax} (R/I)$.

$0\in\mathrm{Q_iMax}( R/I)$. Then, $R/I$ is a field and $\mathrm{L}_R(R/I)=1$ with $|\mathrm{V}_{R/I}(0)|=1 $.

$0\in\mathrm{Q_dMax}( R/I)$. Then $0:=M_1\cap M_2$,  intersection of two (distinct) maximal ideals $M_1,M_2$ of $R/I$, so that $|\mathrm{V}_ {R/I}(0)|=2$ and $\mathrm{L}_R(R/I)=2$ because $\mathrm{L}_R(R/M_i)= 1$ and $M_i\succ 0$ for $i\in\mathbb N_2$ by \cite[Remark 2]{AKKT}. 

$0\in\mathrm{Q_rMax}( R/I)$. 
 Let $M\in\mathrm{Max}(R/I)$ be such that $M^2= 0\subset M$, so that $|\mathrm{V}_{R/I}(0)|=1$ since $0$ is $M$-primary by \cite[Proposition 3]{AKKT}. Moreover, $\mathrm{L}_R(R/M)=1$ and $M\succ 0$ by \cite[Corollary 2, p.237]{ZS} since $M=(R/I)a$ for any $a\in M\setminus\{0\}$. Then, $\mathrm{L}_R(R/I)=2$.

(4)$\Rightarrow$(5): We have $1\leq\mathrm{L}_R(R/I)\leq 2$ and $|\mathrm{V}_{R/I}(0)|\leq 2$. 

If $\mathrm{L}_R(R/I)=1$, then $0$ is the only proper ideal of $R/I$, whence $\mathcal O(R/I)=1$. 

If $\mathrm{L}_R(R/I)=2$ and $|\mathrm{V}_{R/I}(0)|=2$, we claim that $ R/I$ has two maximal ideals $M_1$ and $M_2$. Otherwise, there exists a non maximal prime ideal $P$ contained in a maximal ideal $M$. In this case, we have the chain $0\subseteq P\subset M\subset R/I$ giving $P=0 $ since $\mathrm{L}_R(R/I)=2$, a contradiction by Proposition \ref{2.7} because in this case, $M\succ P$. Then, $0=M_1\cap M_2$ because of $\mathrm{L}_R(R/I)=2$. So $0,M_1$ and $M_2$ are the only proper ideals of $R/I$ and $\mathcal O(R/I)=3$ with $|\mathrm{Max}(R/I)|=2$. Conversely, if $|\mathrm{Max}(R/I)|=2$, we get that $R/I$ has two maximal ideals and we are reduced to the previous situation, so that $\mathcal O(R/I)=3$. 

At last, $\mathrm{L}_R(R/I)=2$ and $|\mathrm{V}_{R/I}(0)|=1$ show that $0$ and $M$ are the only proper ideals of $R/I$ because of the chain $0\subset M\subset R$.  Then, $\mathcal O(R/I)=2$.

(5)$\Rightarrow$(6): Obvious.

(6)$\Rightarrow$(7): Assume that $\mathrm o(I)\leq 3$ with $\mathrm o(I) =3$ if and only if $|\mathrm{V}_R(I)\cap \mathrm{Max}(R)|= 2$.

If $\mathrm o(I)=1$, then $I\in\mathrm{Q_iMax}( R)$. 

If $\mathrm o(I)=2$, there is only one proper ideal $M$ such that $I\subset M$, so that $M\in\mathrm{Max}(R)$. Let $a\in M\setminus I$, then $I+Ra$ is an ideal of $R$ such that $I\subset I+Ra\subseteq M$. Therefore, $I+Ra =M$, the only proper ideal of $R$ containing strictly $I$. Set $R':= R/I,\ M': =M/I$ and let $\overline a$ be the class of $a$ in $R'$. Then $R'$ is a local ring with $\mathcal O(R')=2$ and maximal ideal $M'= R'\overline a$ which is a finitely generated $R'$-module. This implies that either $M'^2=0\ (*)$ or $M'^2=M'\ (**)$. In case $(**)$, Nakayama Lemma gives that $M'=0$, that is $I=M$, a contradiction and only $(*)$ holds, which means that $a^2 \in I,\ M^2\subseteq I\subset M$ and $I$ is an $M$-primary ideal. To conclude, $I\in\mathrm{Q_rMax}( R)$. 
 
Let $\mathrm o(I)=3$, so that $I$ is contained in two (distinct) maximal ideals $M_1$ and $M_2$ such that $I,M_1$ and $M_2$ are the only proper ideals of $R$ containing $I$ whence $I=M_1\cap M_2\in\mathrm {Q_dMax}( R)$. 
  
(7) $\Rightarrow$(1): If $I\in\mathrm{Q_rMax}(R)$, there is only one proper ideal $M$ such that $I\subset M$. Let $a\in M\setminus I$, so that $I+Ra$ is an ideal of $R$ such that $I\subset I+Ra\subseteq M$. Therefore, $I+Ra =M$, the only proper ideal of $R$ containing strictly $I$ and $M$ is a maximal ideal. If $a\in R\setminus M$, then $Ra$ and $M$ are comaximal, so that $Ra$ and $M^2$ are also comaximal, which implies that $I+Ra=R$ and $I\in\mathrm{QMax}(R)$ by the definition of a quasi-maximal ideal.

If $I\in\mathrm{Q_dMax}( R)$, then $I=M_1\cap M_2$. We infer from \cite[Remark 2]{AKKT} that $I\in\mathrm{QMax}(R)$.

If $I\in\mathrm{Q_iMax}(R)$ then $I\in\mathrm{QMax}(R)$ by \cite[Proposition 3]{AKKT} because we already saw that $\mathrm{Q_iMax}(R)=\mathrm{Max}(R)$. 

The proof of (6)$\Rightarrow$(7) shows that $\mathrm{V}_R(I)\subseteq\mathrm{Max}(R)$ for any type.
 \end{proof}

Recall that a {\it special principal ideal ring} (SPIR) is a principal ideal ring $R$ with a unique nonzero prime ideal $M=Rt$, such that $M$ is nilpotent of index $p>0$.

\begin{proposition}\label{2.5} An ideal $I$ of a ring $R$ is in $\mathrm {Q_rMax}(R)$ if and only if $R/I$ is a SPIR whose maximal ideal is nilpotent of index 2.
\end{proposition}
\begin{proof} Assume that $I\in\mathrm{Q_rMax}( R)$ and let $M\in\mathrm{Max}(R)$ be such that $I$ is $M$-primary. We may remark that the inclusion $M^2\subseteq I\subset M$ given in Definition \ref{2.2} implies that $I$ is $M$-primary. According to Theorem \ref{2.3}, $ R':=R/I$ has only two proper ideals: $0$ and $M':=M/I$ with $M'^2=0$. Then, $R'$ is a local Noetherian ring. Moreover, $M'$ is principal with $M'= R'x$ for any $x\in M'\setminus 0$. Therefore $R'$ is principal by \cite[Proposition 4]{Hun}, so that $R'$ is a SPIR whose maximal ideal is  nilpotent of index 2.

Conversely, assume that $R':=R/I$ is a SPIR whose maximal ideal $M'$ is  nilpotent of index 2. Then, $M'^2=0$ and $R'$ has only two proper ideals: $0$ and $M'$. Hence $I\in\mathrm{QMax}(R)$ by Theorem \ref{2.3} and its proof shows that $I\in\mathrm{Q_rMax}( R)$. 
\end{proof}

\begin{remark}\label{2.4} (1) The condition (5) of Theorem \ref{2.3} asserts that for a quasi-maximal ideal $I$ of a ring $R$, the ring $R/I$ has finitely many ideals. In \cite[Corollary 2.4]{AC}, Anderson and Chun characterize a ring with finitely many ideals. A ring has only finitely many ideals if and only if it is a finite direct product of finite local rings, SPIRs, and fields. If $I\in\mathrm{QMax}(R)$, then $R/I$ has finitely many ideals. Each type of quasi-maximal ideal determines the nature of the ring $R/I$ as follows: If $I \in\mathrm{Q_iMax}(R)$, then $R/I$ is a field. If $I\in\mathrm{Q_dMax}(R)$, that is an intersection of two distinct maximal ideals $M_1$ and $M_2$, then $R/I\cong R/M_1\times R/M_2$ is a product of two fields. If $I\in\mathrm{Q_rMax}(R)$, then $R/I$ is a SPIR by Proposition \ref{2.5}. In any case, $R/I$ is an Artinian ring.
 
We may remark that the condition $\mathrm{V}_R(I)\subseteq  \mathrm{Max}(R)$ is equivalent to $\dim (R/I)=0$.
 
(2) Since a ramified ideal is $M$-primary for a maximal ideal $M$, a product of quasi-maximal ideals is a primary decomposition of ideals.

(3) If $I\in\mathrm{QMax}(R)$, then $(\sqrt I)^2\subseteq I$ holds, with $\sqrt I=I$ in the inert and decomposed cases. In the ramified case, $\sqrt I=M\in\mathrm{Max}(R)$ shows that $(\sqrt I)^2= M^2\subseteq I$ is satisfied by the definition of a ramified ideal. We infer that if $I\in\mathrm{QMax}(R)$,  so is $\sqrt I$. 
\end{remark}

\begin{definition}\label{2.72}We say that an ideal $I$ of a ring $R$ is {\it submaximal} if there exists $M\in\mathrm{Max}(R)$ such that $I\prec M$.
 Therefore,  a submaximal ideal is not maximal. 
 \end{definition}

\begin{proposition}\label{2.71} A proper ideal $I$ of a ring $R$ is submaximal if and only if $M\succ I$ for any $M\in\mathrm{Max}(R)$ such that $I\subset M$.
\end{proposition}
\begin{proof} One implication is obvious. 

Conversely, assume that $I$ is submaximal, so that there exists $M\in\mathrm{Max}(R)$ such that $I\prec M$, whence $I\not\in\mathrm{Max}(R)$. We claim that $M'\succ I$ for any $M'\in\mathrm{Max}(R)$ such that $I\subset M'$. Otherwise, there exists $M'\in\mathrm{Max}(R)$ such that $I\subset M'$ with $M'\nsucc I$, giving that $M'\neq M$. This means that there exists an ideal $J$ of $R$ such that $I\subset J\subset M'$. Then, $I\subseteq M\cap J\subseteq M\cap M'\subset M$. We infer that $I=M\cap J =M\cap M'\ (*)$ since $M\cap M'\neq M$. Therefore, $J=M'$, a contradiction. Indeed, it is enough to localize $(*)$ at $M,M'$ and $N$ for any $N\in\mathrm{Max}(R),\ N\neq M,M'$ to get $I_M=M_M$, and $J_M= M'_M=R_M$ because $J\not\subseteq M$ since $I=J$ if $J\subseteq M$ by $(*)$, a contradiction, $I_{M'}=M'_{M'}=J_{M'},\ M_{M'}=R_{M'}$ and $I_N=M_N =J_N=M'_N=R_N$.
\end{proof}
  
\begin{theorem}\label{2.6} Let $I$ be a non maximal ideal  of $R$. The following statements are equivalent:
\begin{enumerate}
\item $I\in\mathrm{QMax}(R)$.
\item $I$ is submaximal.
\item $I$ is strictly contained only in maximal ideals.
\end{enumerate} 
\end{theorem}
\begin{proof} (1)$\Rightarrow$(2): The different types of quasi-maximal ideals gotten in the proof of Theorem \ref{2.3} show that a quasi-maximal ideal of $R$ which is not maximal is obviously submaximal.

(2)$\Rightarrow$(3): If $I$ is submaximal, $M\succ I$ for any $M\in\mathrm {Max}(R)\cap\mathrm{V}(I)$ by Proposition \ref{2.71}. Hence $I$ is strictly contained only in maximal ideals.

(3)$\Rightarrow$(1): Assume that $I$ is strictly contained only in maximal ideals. We claim that $I$ is contained in at most two proper ideals different from $I$, with exactly two if they are maximal. Let $M\in\mathrm{Max}(R)$ be such that $I\subset M$ and assume that there is an ideal $J\neq M$ such that $I\subset J$. Then $J\in\mathrm{Max}(R)$. But $I\subseteq J\cap M\subset J,M$ shows that $I=J\cap M$ because $J\cap M$ is not maximal. Therefore $I\in\mathrm{Q_dMax}(R)$ by Definition \ref{2.2}. If $I$ is contained in only one (maximal) ideal, Theorem \ref{2.3} says that $I\in\mathrm{QMax}(R)$ because $\mathrm o(I)=2$.
 \end{proof}

This Theorem shows that any proper ideal containing strictly a quasi-maximal ideal is maximal. 

According to \cite[Corollary 8.28]{Lam}, a ring $R$ is a {\it Kasch} ring if any maximal ideal of $R$ has the form $(0:x)$ for some $x\in R$, or equivalently, for any proper ideal $I$ of $R$, then $(0:I)\neq 0$. Therefore $\mathrm{t}(R)=R$ since a regular element of $R$ is necessarily a unit, because contained in no maximal ideal of $R$.  Quasi-maximal ideals allows to construct Kasch rings with the next Proposition.

\begin{proposition}\label{2.8} If $I\in\mathrm{QMax}(R)$, then $R/I$ is a  Kasch ring and $\mathrm{t}(R/I)=R/I$. 
\end{proposition}
\begin{proof} Any maximal ideal of $R/I$ is of the form $N:=M/I$ where $M\in\mathrm{Max}(R)\cap\mathrm{V}(I) $. If $I\in\mathrm{Q_iMax}(R)$, then $N=0=(0:x)$ for any nonzero $x$ of the field $R/I$. If $I\in\mathrm {Q_dMax}(R)$, then $I=M\cap M'$ for some $M'\in\mathrm{Max}(R),\ M' \neq M$. Set $N':=M'/I$. We deduce that $N=(0:x)$ for any $x\in N' \setminus N$. If $I\in\mathrm{Q_rMax}(R)$, then $M^2\subseteq I\subset M$ shows that $N=(0:x)$ for any $x\in N,\ x\neq 0$. So, in any case, $R/I$ is a Kasch ring. Then, $\mathrm{t}(R/I)=R/I$ by the property of Kasch rings.
\end{proof}
 
\begin{remark}\label{2.81} The converse of Proposition \ref{2.8} does not hold in general as we can see in the following example: Let $R$ be a ring and $M\in\mathrm{Max}(R)$ be such that $M\neq M^2$ and $M\nsucc M ^2$. Setting $R':=R/M^2$ and $M':=M/M^2$, we get that $(R',M')$ is a local ring. Let $x\in M',\ x\neq 0$. Then, $M':=(0:x)$, so that $R/M^2$ is a  Kasch ring, but $M^2\not\in\mathrm{QMax}(R)$ because $M\nsucc M^2$.
\end{remark}

Of course, any ring has inert ideals, and any non local ring has decomposed ideals. We are going to show that ramified ideals exist in a ring having a finitely generated maximal ideal. 

\begin{proposition}\label{2.9} If $R$ is a ring which is not a field and such that there exists $M\in\mathrm{Max}(R)$ finitely generated, there exists $I\in\mathrm {QMax}(R)$ such that $I\subset M$.
\end{proposition}
\begin{proof} Set $M=\sum_{i=1}^nRx_i$ and $\mathcal F:=\{I$ ideal of $R\mid I\subset M\}$. Then, $\mathcal F\neq\emptyset$ because $0\in\mathcal F$. We claim that $\mathcal F$ is inductive. Let $\mathcal F'$ be a linearly ordered subset of $\mathcal F$ and set $J:=\cup[I_{\alpha}\mid I_{\alpha}\in\mathcal F']$. Then, $J$ is an ideal of $R$ such that $J\subseteq M$. Assume that $J=M$. Hence, for each $i\in\mathbb{N}_n$, there exists some $I_{\alpha_i}\in\mathcal F'$ such that $x_i\in I_{\alpha_i}$. Let $I_{\alpha_k}$ be the greatest of such $I_{\alpha_i}$. Therefore $I_ {\alpha_k}\in\mathcal F$ with $x_i\in I_{\alpha_k}$ for each $i\in\mathbb{N}_n$, so that $M\subseteq I_{\alpha_k}\subset M$, a contradiction. We get $J\subset M$, giving $J\in\mathcal F$ and $\mathcal F$ has a maximal element $I$ such that $I\subset M$. In particular, $M\succ I$, because, if there exists some ideal $K$ such that $I\subset K\subset M$, we should have $K\in\mathcal F$, a contradiction with the maximality of $I$. Then, $I\in\mathrm{QMax}(R)$ by Theorem \ref{2.6}.
\end{proof}

\begin{corollary}\label{2.10} If $J$ is an ideal of a ring $R$ such that $J\subset M$, where $M\in\mathrm{Max}(R)$ is finitely generated, there exists $I\in\mathrm {QMax}(R)$ such that $J\subseteq  I\subset M$. 
\end{corollary}
\begin{proof} Set $R':=R/J$ and $M':=M/J$, so that $R'$ is not a field. Then, $M'$ is a finitely generated maximal ideal of $R'$ and Proposition \ref{2.9} says that there exists an ideal $K\in\mathrm{QMax}(R)$ of $R'$
 such that $K\subset M'$. We infer from Theorem \ref{2.3} that $0\in\mathrm{QMax}(R'/K)$. Moreover, there exists an ideal $I$ of $R$ such that $K=I/J$ with $I\subset M$. As $R'/K=(R/J)/(I/J)\cong R/I$, the same Theorem shows that $I\in\mathrm{QMax}(R)$ is such that $J\subseteq I\subset M$.
\end{proof}

\begin{corollary}\label{2.20} Let $R$ be a Noetherian ring which is not a field. For any $M\in\mathrm{Max}(R)$, there exists $I\in\mathrm{QMax}(R)$ such that $M\succ I$.
 \end{corollary}
\begin{proof} Use Corollary \ref{2.10}. 
\end{proof}

\begin{proposition}\label{2.12} If $J$ is an $M$-primary ideal of a ring $R$, where $M\in\mathrm{Max}(R)$ is such that $1\leq\mathrm{L}_R(M/J)<\infty$, there exists $I\in\mathrm{Q_rMax}(R)$ such that $J\subseteq I\subset M$. 
\end{proposition}
\begin{proof} Let $J$ be an $M$-primary ideal of a ring $R$, where $M\in\mathrm{Max}(R)$ such that $1\leq\mathrm{L}_R(M/J)=n<\infty$. Then there exists a decreasing sequence of ideals $\{J_i\}_{i=0}^n$ of $R$ such that $J_0=M$ and $J_n=J$ with $J_i\succ J_{i+1}$. Setting $I:=J_1$, we get that $I$ is $M$-primary and $\mathrm{L}_R(M/I)=1$, that is $M\succ I$. Moreover, $J\subseteq I$. Because $\mathrm o(I)=2$ and according to Theorem \ref{2.3}, we get that $I\in\mathrm{Q_rMax}(R)$ with $J\subseteq  I\subset M$.
\end{proof}

\begin{proposition}\label{2.11} If $R$ is a ring and $K:=\cap\{I\in\mathrm {QMax}(R)\}$, then $\mathrm J(R)^2\subseteq\cap[M^2\mid M\in\mathrm{Max}(R)]\subseteq K\subseteq\mathrm J(R)$.
\end{proposition}
\begin{proof} For any $I\in\mathrm{QMax}(R)$, there is $M\in\mathrm {Max}(R)$ such that $I\subseteq M$, and any maximal ideal is in $\mathrm {QMax}(R)$. Then, we have $K\subseteq\cap[M\mid M\in\mathrm{Max}(R)] =\mathrm J(R)$. Moreover, any inert or ramified ideal contains the square of some maximal ideal. If $I\in\mathrm{Q_dMax}(R)$, then $I=M\cap M'=M M'$ for some $M,M'\in\mathrm{Max}(R),\ M'\neq M$. Hence $M^2M'^2=M^ 2\cap M'^2\subseteq M\cap M'=I$, whence $\mathrm J(R)^2\subseteq\cap[M^2\mid M\in\mathrm {Max}(R)]\subseteq K$.
\end{proof}

\begin{remark}\label{2.13} Contrary to the Jacobson radical, the nilradical is not always comparable to $K$ as we can see in the following example. Let $(S,N)$ be a local PID and set $N:=St$ for some $t\in N$. Let $R:=S [X]/(X^3)=S[x]$, where $X$ is an indeterminate and $x$ is the class of $X$ in $R$, so that we have an integral extension $S\subset R$ since $x^3=0$.   We claim that $R$ is a local ring with maximal ideal $M:=Rt+Rx$. Let $M\in\mathrm{Max}(R)$. There exists $P\in\mathrm{Max}(S[X])$ such that $ X^3\in P$, and then $X\in P$, with $M=P/(X^3)$. Moreover, since $R$ is integral over $S$, then $M\cap S\in\mathrm{Max}(S)$, so that $M\cap S=N$. In particular, since $t\in N$, we can write $t=f(X)+X^3g(X)$, with $f(X)\in P$ and $\deg(f)<3$. This implies that $f(X)=t\in P\cap S$, so that $P\cap S =N$. According to \cite[Theorem 6.9]{Pic 17}, we get $P=N[X]+XS[X]=tS[X] +XS[X]$. Therefore $M=Rt+Rx$ and $M$ is the only maximal ideal of $R$, which is a local ring. An application of Proposition \ref{2.11} gives $M^2 \subseteq K\subseteq\mathrm J(R)=M$. As $t^2\in M^2\subseteq K$ is a regular element of $R$, we have $K\not\subseteq\mathrm{Nil}(R)$. Now, let $I:=Rt+Rx^2$. As $I+Rx=Rt+Rx=M$, we deduce that $I\in\mathrm {QMax}(R)$, with $x\not\in I$ and $x\not\in K$ holds, although $x\in\mathrm{Nil}(R)$. Then, $\mathrm{Nil}(R)\not\subseteq K$. To conclude, $K$ and $\mathrm{Nil}(R)$ are not comparable.
\end{remark}

\begin{proposition}\label{2.14} Let $\{R_1,\ldots,R_n\}$ be a finite family of rings and $R:=R_1\times\ldots\times R_n$. The ideals of $\mathrm{QMax}(R)$ are the ideals of the form $M_i\times M_j\times\Pi_{k\neq i,j}R_k$, where $M_i$ (resp. $M_j$) $\in\mathrm{Max}(R_i)$ (resp. $\in\mathrm {Max}(R_j)$) and $I_i\times\Pi_{k\neq i}R_k$, where $I_i\in\mathrm{QMax}(R_i)$  (after some reordering).
\end{proposition}
\begin{proof} Let $I\in\mathrm{QMax}(R)$. Then $\mathrm o(I)\leq 3$ with $\mathrm o(I)=3$ if and only if $|\mathrm{V}_R(I)\cap\mathrm{Max}(R)|= 2$ by Theorem \ref{2.3}. 

If $I$ is contained in two maximal ideals of $R$, then $I$ is an intersection of these two maximal ideals, and is of the form $M_i\times M_j\times\Pi_{k\neq i,j}R_k$, where $M_i$ (resp. $M_j$) $\in\mathrm{Max}(R_i)$ (resp. $\in\mathrm {Max}(R_j)$) after some reordering. 

If $I$ is contained in only one maximal ideal $M$ of $R$, for instance $M= M_i\times S$, where $S:=\Pi_{k\neq i}R_k$, with $R=R_i\times S$ (after reordering), then either $I=M$, or $I\in\mathrm{Q_rMax}(R)$, and then is $ M$-primary and satisfies $M^2=M_i^2\times\Pi_{k\neq i}R_k\subseteq I\subset M=M_i\times\Pi_{k\neq i}R_k$. In both cases, $I$ is of the form $ I_i\times\Pi_{k\neq i}R_k$. Let $\psi:R\to R_i$ be the canonical surjection with kernel $0\times S$. According to \cite[Proposition 2]{AKKT}, we get that $\psi(I)= I_i\in\mathrm{QMax}(R_i)$. 
\end{proof}

A ring $R$ is called a {\it $\mathrm{Max}$-ring} if any nonzero $R$-module has a maximal submodule. An ideal $I$ of $R$ is {\it $\mathrm T$-nilpotent} if for each sequence $\{r_i\}_{i=1}^{\infty}$ in $I$, there is some positive integer $k$ with $r_1\cdots r_k=0$. 

\begin{proposition}\label{2.15} For any $M\in\mathrm{Max}(R)$, where $ R$ is a $\mathrm{Max}$-ring which is not a field, there exists a quasi-maximal ideal $I\subset M$. 
\end{proposition}
\begin{proof} Let $M\in\mathrm{Max}(R)$. Then $M$ is an $R$-module whose $R$-submodules are ideals of $R$. As any maximal submodule of $M$ is a submaximal ideal, and according to Theorem \ref{2.6}, it follows that  there exists a quasi-maximal ideal $I\subset M$. 
\end{proof}

\begin{corollary}\label{2.16} Let $R$ be a $\mathrm J$-regular ring which is not a field and such that $\mathrm J(R)$ is $\mathrm T$-nilpotent. Then, for any $M\in\mathrm{Max}(R)$, there exists a quasi-maximal ideal $I\subset M$. 
\end{corollary}
\begin{proof} By \cite[Theorem, p.1135]{Ham}, $R$ is a $\mathrm {Max}$-ring. An application of Proposition \ref{2.15} shows that for any $M \in\mathrm{Max}(R)$,  there exists a quasi-maximal ideal $I\subset M$. 
\end{proof}

Obviously, maximal ideals are inert and in every non local ring the  intersection of two maximal ideals is decomposed. The next corollary gives a characterization of maximal ideals covering ramified ideals.

\begin{corollary}\label{2.19} Let $R$ be a ring and $M\in\mathrm{Max}(R)$. There exists $I\in\mathrm{Q_rMax}(R)$ with $I\subset M$ if and only if $M^2\neq M$. 
\end{corollary}
\begin{proof} If $M^2=M$, there cannot exists some ideal $I$ of $R$ such that $M^2\subseteq I\subset M$. This gives one implication.

Conversely, assume that $M^2\neq M$. Set $R':=R/M^2,\ M':=M/M^2$ and let $f:R\to R'$ be the surjective natural map, where $R'$ is not a field. Then, $(R',M')$ is a local ring with $M'^2=0$ and $M'=\mathrm J(R')$ is $\mathrm T$-nilpotent. Moreover, $R'/\mathrm J(R')=R'/M'$ is a field, giving that $R'$ is $\mathrm J$-regular. According to Corollary \ref{2.16}, there exists a quasi-maximal ideal $I'\subset M'$, which is necessarily in $\mathrm{Q_rMax}(R')$ since $R'$ is local. Set $I:=f^{-1}(I')$ which is in $\mathrm{QMax}(R)$ by \cite[Proposition 2]{AKKT} with $M^2\subseteq I$.  
 Hence, $I\subset M$ and $I$ is $M$-primary, whence in $\mathrm {Q_rMax} (R)$. 
\end{proof}

The following Proposition gives examples of $\mathrm{Max}$-rings.

\begin{proposition} \label{2.17} A ring $R$ which is not a field and such that $0\in\mathrm {QMax}(R)$ is a $\mathrm{Max}$-ring. 
\end{proposition}
\begin{proof} Since $0\in\mathrm{QMax}(R)$, either $R$ is local with maximal ideal $M$, or $R$ has two maximal ideals, $M_1$ and $M_2$, with $0= M_1 \cap M_2$. 

If $(R,M)$ is local, then $M=\mathrm J(R)$ and $R/M$ is a field, which is  regular. Moreover, $M^2=0$ in both cases, $0$ inert or ramified, so that $\mathrm J(R)$ is $\mathrm T$-nilpotent. We infer that $R$ is a $\mathrm {Max}$-ring by \cite[Theorem, p.1135]{Ham}.

If $R$ has two maximal ideals, $M_1$ and $M_2$, then $\mathrm J(R)= M_1\cap M_2=0$ is $\mathrm T$-nilpotent. Therefore $R=R/\mathrm J(R)\cong R/M_1\times R/M_2$, a product of two fields, giving that $R$ is regular. The same reference implies that $R$ is a $\mathrm {Max}$-ring. 
\end{proof} 

Proposition \ref{2.17} has no converse. It is enough to consider $R:=\mathbb Z/30\mathbb Z$ which is a $\mathrm {Max}$-ring by \cite[Theorem, p.1135]{Ham} because $J(R)=0$ is $\mathrm T$-nilpotent and $R=R/\mathrm J(R)$ is a product of three fields, giving that $R$ is regular. But $0\not\in\mathrm{QMax}(R)$, since an intersection of three maximal ideals of $R$.

\section{Quasi-maximal ideals through extensions}
  
In \cite[Proposition 2]{AKKT}, the behavior of quasi-maximal ideals through surjective ring morphisms is considered. We extend their results to arbitrary ring morphisms and  it is enough to look at  extensions. 

We begin by giving two lemmas linking the radical of an ideal and its extension in a ring extension. 

\begin{lemma}\label{3.25}Let $R\subseteq S$ be an LO extension and $I$ an ideal of $R$. Then   $\sqrt[R]I=\sqrt[S]{IS}\cap R$.
\end{lemma}

\begin{proof} We know that $\sqrt[R]I=\cap_{P\in\mathrm{V}_R(I)}P$. Let $ P\in\mathrm{V}_R(I)$. Since $R\subseteq S$ has LO, there exists $Q\in\mathrm{Spec}(S)$ such that $P=Q\cap R$. It follows that $IS\subseteq Q$, which implies $\sqrt[S]{IS}\subseteq Q$, giving $\sqrt[S]{IS}\cap R\subseteq Q\cap R=P$. Then, $\sqrt[S]{IS}\cap R\subseteq \sqrt[R]I$. 

Now, let $Q\in\mathrm{V}_S(IS)$ and set $P:=Q\cap R$. Therefore, $I\subseteq IS\cap R\subseteq Q\cap R=P$, so that $P\in\mathrm{V}_R(I)$ and $\sqrt[R]I\subseteq P=Q\cap R$, that is $\sqrt[R]I$ is contained in any $Q\in\mathrm{V}_S(IS)$, and then in $\sqrt[S]{IS}$. To conclude, $\sqrt[R]I =\sqrt[S]{IS}\cap R$.
\end{proof}

\begin{lemma}\label{3.26}Let $R\subseteq S$ be an extension and $I$ an ideal of $R$ such that $I=IS\cap R$. Then $\sqrt[R]I=\sqrt[S]{IS}\cap R=\sqrt[S]{\sqrt[R]IS}\cap R$.
\end{lemma}

\begin{proof} Thanks to \cite[Theorem 9, (16) and (22), p. 148]{ZS}, we have $\sqrt[R]{IS\cap R}=\sqrt[S]{IS}\cap R=\sqrt[R]IS\cap R$, giving the first equality. Since $\sqrt[R]IS=\sqrt[S]{IS}=\sqrt[S]{\sqrt[R]IS}$, this implies that $\sqrt[R]I=\sqrt[S]{\sqrt[R]IS}\cap R$.
\end{proof}

\begin{corollary}\label{3.27}Let $R\subseteq S$ be an extension and $I$ be an ideal of $R$ such that $\sqrt[R]I=\sqrt[S]{IS}\cap R$. Then $\sqrt[R]I=\sqrt[S]{\sqrt[R]IS}\cap R$.
\end{corollary}

\begin{proof} Set $J:=\sqrt[R]I=\sqrt[S]{IS}\cap R$ by assumption. Then, the proof of Lemma \ref{3.26} gives that $J=\sqrt[R]IS\cap R=JS\cap R$. Applying Lemma \ref{3.26} to $J$, we get that $J=\sqrt[R]J=\sqrt[S]{JS}\cap R=\sqrt[S]{\sqrt[R]IS}\cap R$.
\end{proof}

In the following, we will often meet extensions $R\subseteq S$ such that $MS\in\mathrm{Max}(S)$ for some $M\in\mathrm{Max}(R)$. This happens, in particular, for Nagata extensions $R\subseteq R(X)$.

\begin{proposition}\label{3.28} Let $R\subseteq S$ be an extension and $I\in\mathrm{QMax}(R)$ be such that $MS\in\mathrm{Max}(S)$ for any $M\in\mathrm{Max}(R)\cap\mathrm{V}(I)$. Then $IS\in\mathrm{QMax}(S)$ and $\sqrt [R]I=\sqrt[S]{IS}\cap R$. Moreover, $I$ and $IS$ are homotypic 
 and $I=IS\cap R$ in all cases unless $I\in\mathrm{Q_rMax}(R)$ with $IS=\sqrt[S]{IS}$. 
\end{proposition}
\begin{proof} If $I\in\mathrm{Q_iMax}(R)$, then $I\in\mathrm{Max}(R)$, so that $IS\in\mathrm{Max}(S)=\mathrm{Q_iMax}(S)$. Moreover, $I=\sqrt[R]I=IS\cap R=\sqrt[S]{IS}\cap R$.

If $I\in\mathrm{Q_dMax}(R)$, then $I=M_1\cap M_2=M_1M_2$, with $M_i\in\mathrm{Max}(R)$ for $i\in\mathbb N_2$. Setting $N_i:=M_iS\in\mathrm {Max}(S)$ (by assumption), we get $IS=M_1M_2S=(M_1S)(M_2S)=N_1 N_2=N_1\cap N_2$, so that $IS\in\mathrm{Q_dMax}(S)$. Moreover, $I= M_1\cap M_2=M_1M_2=\sqrt[R]I=(N_1\cap N_2)\cap R=(M_1S\cap M_2S)\cap R=IS\cap R=\sqrt[S]{IS}\cap R$.

If $I\in\mathrm{Q_rMax}(R)$, there exists $M\in\mathrm{Max}(R)$ such that $M\succ I$ with $M^2\subseteq I\subset M\ (*)$. Then $M=\sqrt[R]I=\sqrt[S]{IS}\cap R$. Setting $N:=MS\in\mathrm{Max}(S)$, we get $M^2S= N^2\subseteq IS\subseteq MS=N=\sqrt[S]{IS}$. Moreover, for any $x\in M\setminus I$, we have $M=I+Rx$ giving $N=(I+Rx)S=IS+Sx$. If $IS\neq\sqrt[S]{IS}$, we get that $x\in N\setminus IS$. Otherwise, $Sx\subseteq IS$ gives $IS=\sqrt[S]{IS}=N$, a contradiction. In particular, $x\in(N\cap R) \setminus(IS\cap R)$. Then, $N\succ IS$ by \cite[Corollary 2, p.237]{ZS}, that is $IS\in\mathrm{Q_rMax}(S)$. It follows that $I\subseteq IS\cap R\subset N\cap R=M$, so that $I=IS\cap R$. If $IS=\sqrt[S]{IS}=N$, then $IS\in\mathrm{Max}(S)=\mathrm{Q_iMax}(S)$. In particular, $I\subset M=IS\cap R$. Moreover, $M=\sqrt[R]I=MS\cap R=\sqrt[S]{IS}\cap R$ holds in both cases.

In any case, $IS\in\mathrm{QMax}(S)$ and is homotypic to $I$ unless $I\in\mathrm{Q_rMax}(R)$ and $IS=\sqrt[S]{IS}$, in which case $IS\in\mathrm {Q_iMax}(S)$. 
\end{proof}

\begin{corollary}\label{3.2} Let $R\subseteq S$ be an extension and $I\in \mathrm{Q_rMax}(R)$ be such that $\sqrt[R]IS\in\mathrm{Max}(S)$. Then $I=IS\cap R$ if and only if $IS\in\mathrm {Q_rMax}(S)$. 
\end{corollary}
\begin{proof} Since $I\in\mathrm{Q_rMax}(R)$, then $M:=\sqrt[R]I\in\mathrm{Max}(R)$ and the assumption gives that $N:=MS\in\mathrm{Max}(S)$.  According to the proof of Proposition \ref{3.28}, if $IS\neq\sqrt[S]{IS}=N$, then $IS\in\mathrm{Q_rMax}(S)$ with $I=IS\cap R$ and if $IS=\sqrt [S]{IS}=N$, then $IS\cap R=M$. Therefore, $I=IS\cap R$ if and only if $IS\in\mathrm{Q_rMax}(S)$. 
\end{proof}

According to Proposition \ref{2.71}, a non maximal quasi-maximal ideal is covered by a maximal ideal. Thanks to \cite[Corollary 1, p. 237]{ZS}, Proposition \ref{3.35} shows how, under some assumption, the covering property is kept under extensions.

 Before, we give the following improvement of \cite[Lemma 2.4]{DPP2} which will be useful in this paper.

\begin{lemma}\label{3.46}Let $R\subseteq S$ be an extension and $M\in\mathrm{Max}(R)\setminus\mathrm{Supp}(S/R)$. Then the following properties hold:
\begin{enumerate}
\item$M':=MS$ is the unique ideal of $\mathrm{Spec}(S)$ which lies above $M$. 
\item $M'\in\mathrm{Max}(S)$. 
\end{enumerate}
\end{lemma}

\begin{proof} (1) is in \cite[Lemma 2.4]{DPP2} because $R_M=S_M$.

(2) Assume that $M'\not\in\mathrm{Max}(S)$. Then there exists $Q\in\mathrm{Max}(S)$ such that $M'\subset Q\subset S$, which implies $M=M' \cap R\subseteq Q\cap R\subseteq S\cap R=R$. Since $Q\cap R\neq R$ and $M\in\mathrm{Max}(R)$, we get $M=Q\cap R$ giving $M'=Q$ by (1), a contradiction, so that $M'\in\mathrm{Max}(S)$.
\end{proof}

\begin{proposition}\label{3.35}Let $R\subseteq S$ be an extension.\begin{enumerate} 
\item Let $I$ and $J$ be ideals of $R$ such that $I\prec J$. Assume that $\mathrm{V}(I)\cap\mathrm{Supp}(S/R)=\emptyset$. Then $IS\prec JS$ and $IS\cap R=I$ (resp. $JS\cap R=J$).
\item Let $I'$ and $J'$ be ideals of $S$ such that $I'\prec J'$. Assume that $\mathrm{V}(I'\cap R)\cap\mathrm{Supp}(S/R)=\emptyset$. Then $I'\cap R\prec J'\cap R$ and $(I'\cap R)S=I'$ (resp. $(J'\cap R)S=J'$).
\end{enumerate}
\end{proposition}
\begin{proof} Let $I$ and $J$ be ideals of $R$ such that $I\prec J$. According to \cite[Corollary 1, p. 237]{ZS}, there exists $M\in\mathrm{Max}(R)$ such that $MJ\subseteq I\ (*)$ and $x\in J$ such that $J=I+Rx\ (**)$, with $I\subseteq M$. We claim that this $M$ is unique. Otherwise, there exists $M'\in\mathrm{Max}(R),\ M'\neq M$ such that $M'J\subseteq I\ (***)$. Since $M+M'=R$, we get $J\subseteq I$ by $(*)$ and $(***)$, a contradiction with $I\prec J$.   

(1) Moreover, assume that $\mathrm{V}(I)\cap\mathrm{Supp}(S/R)=\emptyset$. Then $N:=MS\in\mathrm{Max}(S)$ by Lemma \ref{3.46} because $M\not\in\mathrm{Supp}(S/R)$. Hence $(*)$ implies $(MS)(JS)=N (JS)\subseteq IS$. By $(**)$, we have $JS=IS+Sx$. We show that $JS\neq IS$. Since $I\neq J$, there exists $M'\in\mathrm{Max}(R)$ such that $J_ {M'}\neq I_{M'}$. In particular, $M'\in\mathrm{V}(I)$, because if not, $M'\not \in\mathrm{V}(J)$, giving $J_{M'}=I_{M'}=R_{M'}$, a contradiction. But $M' \not\in\mathrm{Supp}(S/R)$, so that $R_{M'}=S_{M'}$, giving $J_{M'}= (JS) _{M'}\neq I_{M'}=(IS)_{M'}$. Therefore, $JS\neq IS$, and $IS\prec JS$ by \cite[Corollary 1, p. 237]{ZS}. Of course, $I\subseteq IS\cap R$.
 
Let $M''\in\mathrm{Max}(R)\setminus\mathrm{Supp}(S/R)$. Therefore, $R_ {M''}=S_{M''}$ giving $I_{M''}=(IS)_{M''}=(IS\cap R)_{M''}$.

Now, let $M''\in\mathrm{Max}(R)\cap\mathrm{Supp}(S/R)$. Then, $M''\not\in\mathrm{V}(I)$. Hence, $I_{M''}=R_{M''}$ implies $(IS\cap R)_{M''}=(RS\cap R)_{M''}=R_{M''}=I_{M''}$. 
 
 To conclude, $I=IS\cap R$ and $J=JS\cap R$ with a similar proof.
 
(2) Let $I'$ and $J'$ be ideals of $S$ such that $I'\prec J'$. Assume that $\mathrm{V}(I'\cap R)\cap\mathrm{Supp}(S/R)=\emptyset$. Always by  \cite[Corollary 1, p. 237]{ZS}, there exists $N\in\mathrm{Max}(S)$ such that $NJ'\subseteq I'$ with $I'\subseteq N$. Set $M:=N\cap R,\ I:=I'\cap R$ and $J:=J'\cap R$. Then, $MJ\subseteq NJ'\cap R\subseteq I'\cap R=I\ (*)$. We claim that $I'= IS$ and $J'=JS$. Obviously, $IS\subseteq I'$ and $JS\subseteq J'$. We are going to show that these inclusions are equalities by localizing at each maximal ideal of $R$. 
 
Let $M'\in\mathrm{Max}(R)\setminus\mathrm{Supp}(S/R)$. Therefore, $R_ {M'}=S_{M'}$ giving $I_{M'}=(IS)_{M'}=(I'\cap R)_{M'}=I'_{M'}\cap R_{M'}=I' _{M'}$. In the same way, we get $J'_{M'}=(JS)_{M'}$. 

Now, let $M'\in\mathrm{Max}(R)\cap\mathrm{Supp}(S/R)$. Then, $M'\not\in\mathrm{V}(I)$, and a fortiori, $M'\not\in\mathrm{V}(J)$ since $I\subseteq J$. Hence, $I_{M'}=J_{M'}=R_{M'}$ which implies $(IS)_{M'}=(J S)_{M'}=(RS)_{M'}=S_{M'}=I'_{M'}=J'_{M'}$, since $IS\subseteq I'$ (resp.  $ JS\subseteq J'$). 

To conclude, $I'=IS$ and $J'=JS\ (**)$ which gives $I\neq J$ since $I'\neq J'$. Let $x\in J\setminus I$, so that $x\in J'\setminus I'$ because $I=I'\cap R$. Hence, $J'=I'+Sx\ (***)$ because $I'\prec J'$. Localizing $(***)$ at each maximal ideal of $R$ and using the equalities of the previous paragraph, we get the following:
 
If $M'\in\mathrm{Max}(R)\setminus\mathrm{Supp}(S/R)$, then $J'_{M'}=(I'+ Sx)_{M'}=(JS)_{M'}=J_{M'}=I'_{M'}+S_{M'}(x/1)=I_{M'}+R_{M'}(x/1)=(I+Rx)_ {M'}$.
  
If $M'\in\mathrm{Max}(R)\cap\mathrm{Supp}(S/R)$, then $M'\not\in\mathrm{V}(J)$ as we have already seen. It follows that $I_{M'}=J_{M'}=R_{M'}=I_{M'}+R_{M'}(x/1)=(I+Rx)_{M'}$.
 
Combining the two cases, $J=I+Rx$ and $I\prec J$ hold, always by \cite[Corollary 1, p. 237]{ZS}.
\end{proof}

In Remarks \ref{4.01}(2) and \ref{4.4}(1), we give examples where results of Proposition \ref{3.35} do not hold because the assertion $\mathrm{V}(I'\cap R)\cap\mathrm{Supp}(S/R)$

\noindent$=\emptyset$ (resp. $\mathrm{V}(I)\cap\mathrm{Supp}(S/R)=\emptyset$) is not satisfied. 

\begin{corollary}\label{3.3} Let $R\subseteq S$ be an extension and $I\in\mathrm{QMax}(R)$, such that $\mathrm{V}(I)\cap\mathrm{Supp}(S/R)=\emptyset$. Then $I=IS\cap R,\ IS\in\mathrm{QMax}(S)$ and is homotypic to $I$ with $IS\prec N:=MS\in\mathrm{Max}(S)$ for such $M\in\mathrm{V}(I)\cap \mathrm{Max}(R)$ when $I$ is submaximal in $R$.
\end{corollary}
\begin{proof} If $I\in\mathrm{Max}(R)$, then $IS\in\mathrm{Max}(S)$ and $I=IS\cap R$ by Lemma \ref{3.46}. 

Now, assume that $I$ is submaximal and $I\prec M\in\mathrm{Max}(R)$. Since $M\not\in\mathrm{Supp}(S/R)$, it follows that $N:=MS\in\mathrm {Max}(S),\ IS\prec N$ and $I=IS\cap R$ thanks to Proposition \ref{3.35}. Therefore, $IS$ is submaximal in $S$. Now, Theorem \ref{2.6} shows that $IS\in\mathrm{QMax}(R)$, and more precisely, $IS\in\mathrm{Q_dMax}(S)\cup\mathrm{Q_rMax}(S)$. If $I\in\mathrm{Q_dMax}(R)$, that is $I=M_1 \cap M_2$, we get $IS=(M_1S)\cap(M_2S)$, so that $IS\in\mathrm {Q_dMax}(S)$. If $I\in\mathrm{Q_rMax}(R)$, then $M^2\subseteq I$ implies $(MS)^2\subseteq IS$ and $IS\in\mathrm{Q_rMax}(S)$ because a primary ideal. To conclude, in any case, $IS$ is quasi-maximal homotypic to $I$.  
\end{proof}

\begin{proposition}\label{3.4} Let $R\subseteq S$ be an extension satisfying LO and INC (for example, an integral extension) and $I\in\mathrm{Q_iMax}(R)\cup\mathrm{Q_dMax}(R)$. Then, there exists $J\in\mathrm{QMax}(S)$ homotypic to $I$  such that $I=J\cap R$. 
\end{proposition}
\begin{proof} If $I\in\mathrm{Q_iMax}(R)$, there exists $J\in\mathrm{Max}(S)$ such that $I=J\cap R$, so that $J\in\mathrm{Q_iMax}(S)$.  

If $I\in\mathrm{Q_dMax}(R)$, there exist $M_1,M_2\in\mathrm{Max}(R)$ such that $I=M_1\cap M_2$ with $M_1\neq M_2$. For each $i\in\mathbb N _2$, there exists $N_i\in\mathrm{Max}(S)$ such that $M_i=N_i\cap R$ with $N_1\neq N_2$. Set $J:=N_1\cap N_2$. Then, $J\in\mathrm {Q_dMax}(S)$. Moreover, $I=M_1\cap M_2=(N_1\cap R)\cap (N_2\cap R)=(N_1\cap N_2)\cap R=J\cap R$.
\end{proof}

\begin{lemma}\label{3.53} Let $f:R\to S$ be a ring morphism, $J$ an ideal of $S$ and $I:=f^{-1}(J)$. Then, $I=f^{-1}(I\cdot S)$.
 \begin{enumerate}
\item There exists an ideal $K$ of $S$ maximal in $\mathcal F:=\{K'$ ideal of $S\mid J\subseteq K',\ I=f^{-1}(K')\}$. Such an ideal $K$ is called a Max-upper of $I$ containing $J$.
\item If $N\in\mathrm{Max}(S)\cap\mathrm{V}_S(J)$, there exists an ideal $K_1$ of $S$ maximal in $\mathcal F':=\{K'$ ideal of $S\mid J\subseteq K' \subseteq N,\ I=f^{-1}(K')\}$. Such an ideal $K_1$ is called a Max-upper of $I$ containing $J$ and contained in $N$.
\end{enumerate}
 \end{lemma}
\begin{proof} Since $f(I)\subseteq J$, we get $f(I)\subseteq f(I)S=I\cdot S\subseteq J$, so that $I\subseteq f^{-1}(I\cdot S)\subseteq f^{-1}(J)=I$, whence $I=f^{-1}(I\cdot S)$. 

(1) We have $\mathcal F\neq\emptyset$ because $J\in\mathcal F$. We claim that $\mathcal F$ is inductive. Let $\mathcal F'$ be a linearly ordered subset of $\mathcal F$ and set $K'':=\cup[K_{\alpha}\mid K_{\alpha}\in\mathcal F']$. Obviously, $K''\in\mathcal F$ and $\mathcal F$ has a maximal element $K$. 

(2) We use the previous proof with $\mathcal F':=\{K'$ ideal of $S\mid J\subseteq K'\subseteq N,\ I=f^{-1}(K')\}$. Then, $\mathcal F'$ has a maximal element $K_1$. 
\end{proof}

\begin{proposition}\label{3.52} Let $R\subseteq S$ be an extension, $J$ an ideal of $S$ and $I:=J\cap R$. Assume that $I\prec M$ for some $M\in\mathrm{Max}(R)$ and $N:=MS\in\mathrm{Max}(S)$ with $J\subseteq N$. Then, there exists a Max-upper $K$ of $S$ of $I$ containing $J$ and contained in $N$. Moreover, $K\in\mathrm{QMax}(S)$ and the next 
  properties hold:
\begin{enumerate} 
\item If $I\in\mathrm{Q_rMax}(R)$, then $K=IS\in\mathrm{Q_rMax}(S)$. 
\item Let $I\in\mathrm{Q_dMax}(R)$ with $I=M_1\cap M_2,\ M_i\in\mathrm {Max}(R)$ for $i\in\mathbb N_2$. Set $N_i:=M_iS$ for $i\in\mathbb N_2$, and assume that $N_1\in\mathrm{Max}(S)$. Then $N_2\neq S$ and we have two cases: 
\begin{enumerate}
\item If $N_2\in\mathrm{Max}(S)$, then $K=IS\in\mathrm{Q_dMax}(S)$.  
\item If $N_2\not\in\mathrm{Max}(S)$, then $IS\subset K$ and $IS\not\in \mathrm{QMax}(S)$.
\end{enumerate}
\end{enumerate}
\end{proposition}
\begin{proof}   
Assume that $I\prec M$ for some $M\in\mathrm{Max}(R)$ such that $N:=M S\in\mathrm{Max}(S)$ with $J\subseteq N$. Then, $I\in\mathrm{Q_rMax}(R)\cup\mathrm{Q_dMax}(R)$. According to Lemma \ref{3.53}, we get that $I=IS\cap R$ and there exists an ideal $K$ of $S$ Max-upper of $I$ containing $J$ and contained in $N$. We claim that $K\in\mathrm{QMax}(S)$. First, $K\subset N$ since $I\prec M$. Assume that $K\nprec N$. Hence, there exists some ideal $K''$ of $S$ such that $K\subset K''\subset N$, so that $I\subseteq K''\cap R\subseteq M=N\cap R$. But $I\prec M$ gives that either $I=K''\cap R\ (*)$ or $K''\cap R=M\ (**)$. In case $(*)$, we have $K''\in\mathcal F$, a contradiction with the maximality of $K$. In case $(**)$, we have $M\subseteq K''$, giving $N=MS\subseteq K''$, also a contradiction. To conclude, $K\prec N$, whence $K\in\mathrm{QMax}(S)$.

(1) If $I\in\mathrm{Q_rMax}(R)$, $M$ is the only maximal ideal of $R$ with $I\subset M$ and $MS\in\mathrm{Max}(S)$. As $IS\subseteq K\subset N= MS$, we have $IS\neq\sqrt[S]{IS}$. According to Corollary \ref{3.2}, 
 $IS\in\mathrm{Q_rMax}(S)$ and $K=IS$. 

(2) Assume that $I\in\mathrm{Q_dMax}(R)$ and set $I=M_1\cap M_2,\ M_i\in\mathrm{Max}(R)$ for $i\in\mathbb N_2$, with $N_i:=M_iS$. Set, for instance, $M=M_1$, so that $N_1=N\in\mathrm{Max}(S)$. Then, $IS=M_1 M_2S=NN_2$. As $M_1$ and $M_2$ are comaximal in $R$, so are $N$ and $N_2$ in $S$. 

First, if $N_2=S$, then $I=M_1\cap M_2=M_1M_2$ implies $IS =M_1S=N$ which shows that $I=IS\cap R=N\cap R=M$, a contradiction: this case cannot appear. We  consider the two possibilities for $N_2$.

(a) If $N_2\in\mathrm{Max}(S)$, the same reasoning as for $N$ shows that $IS=N_1\cap N_2$ and $IS=K\in\mathrm {Q_dMax}(S)$. 

(b) Assume that $N_2\not\in\mathrm{Max}(S)$. It follows that $IS= N\cap N_2$. As $K\prec N$, either $K\in\mathrm{Q_dMax}(S)$ ($\alpha$) or $K\in\mathrm {Q_rMax}(S)$ ($\beta$). 

In case ($\alpha$), $K=N\cap N'$, where $N'\neq N$ and $N'\in\mathrm {Max}(S)$ with $IS\subseteq N_2\subset N'$. In this case, $IS\subset K$ and $IS$ is not in $\mathrm{QMax}(S)$.

In case ($\beta$), $N$ is the only maximal ideal of $S$ with $K\subseteq N$. But $IS=NN_2$ shows that $IS$ is contained in at least two maximal ideals of $S$. Then, $IS\subset K$ and $IS$ is not in $\mathrm{QMax}(S)$. 
\end{proof}

\begin{remark}\label{3.54} (1) We may remark that in Proposition \ref{3.52}, $IS=K$ does not always hold, and, even if it holds, $I$ and $IS$ are not always homotypic. Moreover, when $IS=K$ holds, the maximal element $K$ of $\mathcal F$ gotten in Proposition \ref{3.52}(2)(a) is unique since equal to $IS$. 

(2) Let $R\subseteq S$ be an extension, $J$ an ideal of $S$ and $I:=J\cap R$. If $I\in\mathrm{Q_iMax}(R)$,  then $J=IS\in\mathrm{Q_iMax}(S)$ 
 when $IS\in\mathrm{Max}(S)$ since $IS\subseteq J$ always holds.
\end{remark}

We use the results of Olivier about pure extensions \cite{OP}. Recall that a ring morphism $f:R\to S$ is called {\it pure} if $R$ is a pure submodule of $ S$; that is, $R'\to R'\otimes_RS$ is injective for each ring morphism $R\to R'$.  
 
Purity is an universal property and $R\to S$ is pure if and only if $R_P\to S_P$ is pure for each $P\in \mathrm{Spec}(R)$.

A faithfully flat morphism is pure. If $f$ is pure, then $f^{-1}(IS)=I$ for each ideal $I$ of $R$; so that, $f$ is injective. Moreover, a pure ring morphism $R\to S$ verifies the condition (O), {\em i.e.} an $R$-module $M$ is zero if $M\otimes_RS = 0$.

Clearly, if $f:R\subseteq S$ is an extension of rings, $f$ is pure when there is a retract of the extension $R\subseteq S$ in the category of $R$-modules; that is, there is a morphism of $R$-modules $g:S\to R$ such that $g\circ f=\mathrm{Id}_R$, or equivalently $S=R\oplus J$, where $J$ is an $R$-submodule of $S$. For example, for a ring $R$, the extension $R\subseteq R[X]$ has a retract $R[X]\to R$, defined by $X\mapsto 0$.

\begin{theorem}\label{3.5} Let $R\subseteq S$ be a pure extension and $I$ an ideal of $R$. If $MS\in\mathrm{Max}(S)$ for any $M\in\mathrm{Max}(R)\cap\mathrm{V}(I)$, then $I\in\mathrm{QMax}(R)$ if and only if $IS\in\mathrm{QMax}(S)$; in that case $I$ and $IS$ are homotypic.  
\end{theorem}
\begin{proof} Since $R\subseteq S$ is pure, we have $I=IS\cap R$ for any ideal $I$ of $R$. Let $I$ be an ideal of $R$ such that $MS\in\mathrm{Max}(S)$ for any $M\in\mathrm{Max}(R)\cap \mathrm{V}(I)$.
   
If $I\in\mathrm{QMax}(R)$, Proposition \ref{3.28} shows that $IS\in\mathrm {QMax}(S)$. 

Conversely, assume that $IS\in\mathrm{QMax}(S)$ and let $M\in\mathrm{Max}(R)\cap \mathrm{V}(I)$.

If $IS\in\mathrm{Q_iMax}(S)$, this means that $IS=MS$, giving $I=M\in\mathrm{Q_iMax}(R)$. 
  
Assume now that $IS\in\mathrm{Q_dMax}(S)\cup\mathrm{Q_rMax}(S)$, so that $MS\succ IS$ and $M\neq I$. We claim that $M\succ I$. Otherwise, there exists an ideal $K$ of $R$ such that $I\subset K\subset M$ giving $I S\subseteq KS\subseteq MS$. Since $MS\succ IS$, it follows that either $ KS=IS$ giving $K=KS\cap R=IS\cap R=I$, or $KS=MS$ giving $K=KS\cap R=MS\cap R=M$, a contradiction in both cases. Then, $M\succ I$, so that $I$ is a submaximal ideal of $R$ and then in $\mathrm{QMax}(R)$. If these conditions hold, by Proposition \ref{3.28}, $I$ and $IS$ are homotypic.
\end{proof}

\begin{corollary}\label{3.51} Let $R\subseteq S$ be a pure extension such that $\mathrm{Max}(S)=\{MS\mid M\in\mathrm{Max}(R)\}$. An ideal $I$ of $R$ is in $\mathrm{QMax}(R)$ if and only if $IS\in\mathrm{QMax}(S)$. If this condition holds, then $I$ and $IS$ are homotypic.
\end{corollary}
\begin{proof} Use Theorem \ref{3.5}.
\end{proof}

\begin{corollary}\label{3.6} Let $R$ be a semilocal ring with $\mathrm {Max}(R)=\{M_1,\ldots M_n\}$. Set $S:=\prod_{i=1}^nR_{M_i}$. Then, the following properties hold:
\begin{enumerate}
\item $R\subseteq S$ is pure.
\item $\mathrm{Max}(S)=\{M_1S,\ldots, M_nS\}$.
\item An ideal $I$ of $R$ is in $\mathrm{QMax}(R)$ if and only if $IS\in\mathrm{QMax}(S)$. In that case, $I$ and $IS$ are homotypic. 
\end{enumerate} 
\end{corollary}
\begin{proof} (1) comes from \cite[Exemple 1.3.2]{OP}.

(2) Let $N\in \mathrm{Max}(S)$. Then, there exists some $M_i\in \mathrm{Max}(R)$ such that $N=M_iR_{M_i}\times R'$, where $R':=\prod_{j\neq i}R_{M_j}$ (after reordering). But for $j\neq i$, we have $M_iR_{M_j} =R_{M_j}$, so that $N=M_iS$, giving $\mathrm{Max}(S)=\{M_1S,\ldots M_nS\}$.

(3) Use Corollary \ref{3.51}. 
\end{proof}

\begin{corollary}\label{3.7} Let $R\subseteq S$ be an extension such that $f:S\to R$ is a retract of the extension $R\subseteq S$ in the category of $R$-modules.  Then, the next properties hold:
\begin{enumerate}
\item $R\subseteq S$ is pure.
\item If $I\in\mathrm{QMax}(R)$, then $J:=f^{-1}(I)\in\mathrm{QMax}(S)$  with $I=J\cap R$.
\item If $J\in\mathrm{QMax}(S)$ contains $\mathrm{Ker}(f)$, then $I:=f(J)\in\mathrm{QMax}(R)$. 
\end{enumerate} 

In both cases, $J=I+\mathrm{Ker}(f)$ and $I=J\cap R$ hold.
\end{corollary}
\begin{proof} (1) $R\subseteq S$ is pure since there is a retract of the extension $R\subseteq S$ in the category of $R$-modules. Moreover, $f$ is surjective, so that $R\cong S/\mathrm{Ker}(f)$ and $f(x)=x$ for any $x\in R$. In particular, $\mathrm{Ker}(f)\cap R=0$.

(2) Let $I\in\mathrm{QMax}(R)$. According to \cite[Proposition 2 (ii)]{AKKT}, $J:=f^{-1}(I)\in\mathrm{QMax}(S)$ and $J=\{x\in S\mid f(x)\in I\}$. Therefore $J=I+\mathrm{Ker}(f)$ because $f(x)=x$ for any $x\in I$, and $I=J\cap R$ holds. 

(3) Let $J\in\mathrm{QMax}(S)$ with $\mathrm{Ker}(f)\subseteq J$ and set $I: =f(J)$. Then, $f(J)\in\mathrm{QMax}(R)$ by \cite[Proposition 2 (i)]{AKKT}. Moreover, $J=I+\mathrm{Ker}(f)$ and $I=J\cap R$ also hold because $f(x)=x$ for any $x\in I$. 
\end{proof}

Since a faithfully flat morphism is pure, we can use the results we got for pure extensions to faithfully flat extensions. 

\begin{proposition}\label{3.11} If $(R,M)$ is a local ring, $I$ is an ideal of $R$ and $f(X)\in R[X]$ is a monic polynomial, the following properties hold:
\begin{enumerate}
\item $R\subseteq R[X]/(f(X))$ is pure.
\item $I\in\mathrm{QMax}(R)$ if and only if $I[R[X]/(f(X))]\in\mathrm{QMax} (R[X]/$

\noindent$(f(X))$, in which case $I$ and $I[R[X]/(f(X))]$ are homotypic.
\end{enumerate} 
\end{proposition}
\begin{proof} Set $R':=R[X]/(f(X))$.

(1) Since $R\subseteq R'$ is free, it is  faithfully flat, and then pure.

(2) Let $N\in\mathrm{Max}(R[X])$ be such that $f(X)\in N$. According to  \cite[Theorem 6.9]{Pic 17}, we have $M=N\cap R$ and $N=M[X]+f(X)R[X]$. We infer that $N':=N/(f(X))=(M[X]+f(X)R[X])/(f(X))\cong MR'$. We can apply Theorem \ref{3.5}, so that $I\in\mathrm{QMax}(R)$ if and only if $IR' \in\mathrm{QMax}(R')$. If these statements hold, then $I$ and $IR'$ are homotypic.
\end{proof}

\begin{proposition}\label{3.12} Let $R$ be a ring and $f_1,\ldots,f_n\in R$ such that $\{\mathrm{D}(f_i)\mid i\in\mathbb{N}_n\}$ is a partition of $\mathrm{Spec}(R)$ and set $S:=\prod_{i=1}^nR_{f_i}$. An ideal $I$ of $R$ belongs to $\mathrm{QMax}(R)$ if and only if $IS\in\mathrm{QMax}(S)$, in which case $I$ and $IS$ are homotypic. 
\end{proposition}
\begin{proof} We know that $R\subseteq S$ is faithfully flat by \cite[Proposition 3, p.137]{ALCO}. It follows that $R\subseteq S$ is pure. In order to use Corollary \ref{3.51}, we show that $\mathrm{Max}(S)=\{M\mid M\in\mathrm{Max}(R)\}$. Any maximal ideal of $S$ is of the form (after reordering) $M_iR_{f_i}\times\prod_{j\neq i}R_{f_j}$, where $M_i\in\mathrm {Max}(R)\cap\mathrm{D}(f_i)$. Since $\{\mathrm{D}(f_i)\}$ is a partition of $\mathrm{Spec}(R)$, we get that $M_i\not\in\mathrm{D}(f_j)$ for any $j\neq i$, which means that $f_j\in M_i$ for any $j\neq i$, so that $M_iR_{f_j}=R_{f_j}$ for any $j\neq i$. Then, $M_iR_{f_i}\times\prod_{j\neq i}R_{f_j}= M_iR_{f_i}\times\prod_{j\neq i}M_iR_{f_j}=M_i(R_{f_i}\times\prod_{j\neq i}R _{f_j})=M_iS$. Conversely, let $M\in\mathrm{Max}(R)$. Since $\{\mathrm {D}(f_i)\}$ is a partition of $\mathrm{Spec}(R)$, there is some $i\in\mathbb {N}_n$ such that $f_i\not\in M$. The previous proof shows also that $f_j\in M$ for any $j\neq i$ and $MS=MR_{f_i}\times\prod_{j\neq i}R_{f_j}\in\mathrm{Max}(S)$. To conclude, we get that $\mathrm{Max}(S)=\{MS\mid M\in\mathrm{Max}(R)\}$. Then, Corollary \ref{3.51} shows that for an ideal $I$ of $S,\ I\in\mathrm{QMax}(R)$ if and only if $IS\in\mathrm{QMax}(S)$. If these statements hold, $I$ and $IS$ are homotypic. 
\end{proof} 

\begin{proposition} \label{3.13} Let $(R,M)$ be a local ring, $R^h$ be its Henselization and $I$ be an ideal of $R$. Then $I\in\mathrm{QMax}(R)$ if and only if $IR^h\in\mathrm{QMax}(R^h)$, in which case $I$ and $IR^h$ 
are homotypic. 
\end{proposition}
\begin{proof} We know that $R\subseteq R^h$ is faithfully flat, and then pure. Moreover, $R^h$ is a local ring whose maximal ideal is $MR^h$. Then, Corollary \ref{3.51} shows that for an ideal $I$ of $R,\ I\in\mathrm {QMax}(R)$ if and only if $IR^h\in\mathrm{QMax}(R^h)$, in which case $I$ and $I R^h$ are homotypic. 
\end{proof}

For an arbitrary ring $R$, there is an extension $R\subseteq R(X)$, where $R(X)=R[X]_{\Sigma_R}$ and $\Sigma_R$ is the multiplicatively closed subset of all primitive polynomials of $R[X]$. The ring $R(X)$ is called the {\it Nagata} ring of $R$. We are going to look at the transfer of a quasi-maximal ideal of a ring $R$ in its Nagata ring. Recall that $R\subseteq R(X)$ is faithfully flat and $\mathrm{Max}(R(X))=\{M(X)\mid M\in\mathrm{Max}(R)\}$ \cite[Section 3]{DPP3}.

\begin{proposition}\label{2.111} An ideal $I$ of a ring $R$ is in $\mathrm {QMax}(R)$ if and only if $IR(X)\in\mathrm{QMax}(R(X))$. In that case $I$ and $IR(X)$are homotypic. 
\end{proposition}
\begin{proof} We know that $R\subseteq R(X)$ is faithfully flat, and then pure. Moreover, $\mathrm{Max}(R(X))=\{MR(X)\mid M\in\mathrm{Max}(R)\}$. Then, Corollary \ref{3.51} shows that for an ideal $I$ of $R,\ I\in\mathrm {QMax}(R)$ if and only if $IR(X)\in\mathrm{QMax}(R(X))$. If these statements hold, then $I$ and $IR(X)$ are homotypic.
\end{proof}

An extension $R\subseteq S$ is called {\it local} if $\mathrm U(R)= \mathrm U(S)\cap R$. 

\begin{proposition}\label{3.15} Let $(R,M)\subseteq(S,N)$ be a local flat extension such that $N=MS$. An ideal $I$ of $R$ is in $\mathrm{QMax}(R)$ if and only if $IS\in\mathrm{QMax}(S)$, in which case, $I$ and $IS$ are homotypic. 
\end{proposition}
\begin{proof} According to \cite[Proposition 6.6.1]{EGA}, $R\subseteq S$ is faithfully flat, and then pure. Moreover, $\mathrm{Max}(S)=\{N\}=\{MS\}=\{M'S\mid M'\in\mathrm{Max}(R)\}$. Then, Corollary \ref{3.51} shows that for an ideal $I$ of $R,\ I\in\mathrm{QMax}(R)$ if and only if $IS\in\mathrm {QMax}(S)$. If these statements hold, then $I$ and $I S$ are homotypic.
\end{proof}

Most of the results we get for an extension $R\subseteq S$ and the transfer of the property for an ideal $I$ to be quasi-maximal in this section were under the assumption that $I\subseteq M\in\mathrm{Max}(R)$ is such that $MS\in\mathrm{Max}(S)$. This property does not always hold for flat epimorphisms as we see now. 

\begin{proposition}\label{3.16} Let $R\subseteq S$ be a flat epimorphism and $J$ be an ideal of $S$ such that $I:=J\cap R\in\mathrm{QMax}(R)$. Then $J=IS\in\mathrm{QMax}(S)$ and is homotypic to $I$. 
\end{proposition}
\begin{proof} Since $R\subseteq S$ is a flat epimorphism, we deduce that $J=IS$. Moreover, $J\neq S$ since $I\neq R$, so that there exists $N\in\mathrm{Max}(S)$ such that $J\subseteq N$. Then, $I=J\cap R\subseteq N\cap R\in\mathrm{Spec}(R)$ with either $I\in\mathrm{Max}(R)\ (*)$ or $I$ is a submaximal ideal of $R\ (**)$. Case $(*)$ implies that $J\in\mathrm{Max}(S)$. In case $(**)$ we get that $I\subset M:=N\cap R\in\mathrm{Max}(R)$. Using again the fact that $R\subseteq S$ is a flat epimorphism, we have $ N=MS\in\mathrm{Max}(S)$. Then, $J=IS$ is a submaximal ideal of $S$ with $J\subset N$ because $I\prec M$ implies that $J\prec N$. 
 
If $I\in\mathrm{Q_dMax}(R)$, then $I=M M'$, with $M'\in\mathrm{Max}(R),\ M'\neq M$. It follows that $J=IS=MM'S=NM'S\subset N$. We have neither $M'S=S$ since $J\neq N$ nor $MS=M'S$ since $M\neq M'$. Set $N':=M'S$. Then $N'\in\mathrm{Max}(S)$ with $N'\neq N$, which implies $J=N\cap N'\in \mathrm{Q_dMax}(S)$. 
 
If $I\in\mathrm{Q_rMax}(R)$, then $\sqrt[R]I=M$, giving $\sqrt[R]IS=MS\in\mathrm{Max}(S)$. Since $I=IS\cap R$, then $J=IS\in\mathrm{Q_rMax}(S)$ by Corollary \ref{3.2}. 
 
In any case, $IS\in\mathrm{QMax}(R)$ and is homotypic to $I$.
\end{proof}

\begin{proposition} \label{2.18} Let $R$ be a ring such that $R=\sum_{i= 1}^nRf_i$ and $I$ be an ideal contained in some $M\in\mathrm{Max}(R)$ where some $f_i\not\in M$. If $I\in\mathrm{QMax}(R)$, then $IR_{f_i}\in\mathrm{QMax}(R_{f_i})$. 

The converse holds if $I$ is saturated for $R_{f_i}$.

\end{proposition}
\begin{proof} Since $R=\sum_{i=1}^nRf_i$, for any $M\in\mathrm{Max}(R)$, there is some $f_i\not\in M$. Let $I$ be an ideal of $R,\ I\subseteq M$ for some $M\in\mathrm{Max}(R)\cap\mathrm{D}_R(f_i)$. Then, $MR_{f_i}\in\mathrm{Max}(R_{f_i})$ and $M=MR_{f_i}\cap R$. In particular, $f_i\not\in I$. 

Assume first that $I\in\mathrm{QMax}(R)$.  

If $I\in\mathrm{Q_iMax}(R)\cup\mathrm{Q_dMax}(R)$,  so is $IR_{f_i}$ because either maximal or an intersection of two maximal ideals of $R_{f_i}$.  

If $I\in\mathrm{Q_rMax}(R)$, according to Definition \ref{2.2}, $M=\sqrt[R]I$ and $I$ is $M$-primary. We claim that $I=IR_{f_i}\cap R$. We have $I\subseteq IR_{f_i}\cap R$. Let $x\in IR_{f_i}\cap R$. There exist $a\in I,\ b\in R$ and a positive integer $n$ such that $x=a/f_i^n=b/1$. Therefore $af_ i^m=bf_i^k$ for some positive integers $m$ and $k$, so that $bf_i^k\in I$ with $f_i^k\not\in M$, giving $b$ and $x\in I$. Then, $I=IR_{f_i}\cap R$ and an application of Corollary \ref{3.2} shows that $IR_{f_i}\in\mathrm {Q_rMax}(R_{f_i})$.

Conversely, assume that $I$ is saturated for $R_{f_i}$, that is $I=IR_{f_i}\cap R$ and $IR_{f_i}\in\mathrm{QMax}(R_{f_i})$. Since $I\subseteq M$, we have $IR_{f_i}\subseteq MR_{f_i}$, with $MR_{f_i}\in\mathrm{Max}(R_{f_i})$. Set $I':=IR_{f_i},\ M':=MR_{f_i}$ and $R':=R_{f_i}$.

If $I'\in\mathrm{Q_iMax}(R')\cup\mathrm{Q_dMax}(R')$, that is either a maximal ideal or an intersection of two maximal ideals of $R'$, the same holds for $I$ which is then quasi-maximal. 

If $I'\in\mathrm{Q_rMax}(R')$, we have $M'^2\subseteq I'\subset M'$ giving $M^2\subseteq M'^2\cap R\subseteq I=I'\cap R\subset M=M'\cap R$, so that $I$ is $M$-primary by \cite[Corollary 1, p.153]{ZS}. To show that $I\in\mathrm{Q_rMax}(R)$, it is enough to show that $M\succ I$ by Theorem \ref{2.6}. Since $M'\succ I'$, there exists some $t\in R'$ such that $M'=I'+R' t$, with $t\in M'$. Then $t=x/f_i^k$ for some integer $k$ and $x\in M$. We claim that $M=I+Rx$ with $x\in M\setminus I$. First, $x\not\in I$, otherwise we should have $M'=I'$. Of course, $I\subseteq I+Rx\subseteq M\ (*)$. Let $m\in M$, so that $m/1\in M'$ giving $m/1=(a+bx)/f_i^n$, with $a\in I,\ b\in R$ and some integer $n\geq 0$. Then, there exist some positive integers $ t$ and $s$ such that $f_i^tm=af_i^s+bf_i^sx\in I+Rx$, an $M$-primary ideal by $(*)$. Since $f_i^t\not\in M$, we infer that $m\in I+Rx$, so that $M =I+Rx$. Then, $M\succ I$ by \cite[Corollary 2, p.237]{ZS} and $I\in\mathrm{Q_rMax}(R)$. 
\end{proof} 

\begin{proposition} \label{3.161} Let $R$ be a ring. An ideal $I$ of $R$ is in $\mathrm{QMax}(R)$ if and only if $\mathrm{V}(I)\subseteq\mathrm{Max}(R),\ IR_M\in\mathrm{QMax}(R_M)$ for any $M\in\mathrm{V}(I)$ and $|\mathrm{V}(I)|\leq 2$ with $IR_M\in\mathrm{Max}(R_M)$ for any $M\in\mathrm{V}(I)$ when $|\mathrm{V}(I)|=2$, in which case $I$ and $IR_M$ are homotypic when $I\in\mathrm{Q_iMax}(R)\cup\mathrm{Q_rMax}(R)$. If $I\in\mathrm{Q_dMax}(R)$, then $IR_M\in\mathrm{Q_iMax(R_M)}$. 
 \end{proposition}
\begin{proof} Assume that $I\in\mathrm{QMax}(R)$ and let $M\in\mathrm {V}(I)\subseteq\mathrm{Max}(R)$. Then, $M=MR_M\cap R$ with $\{MR_M \}=\mathrm{Max}(R_M)$. According to Proposition \ref{3.52} and Remark \ref{3.54}, we get that $IR_M\in\mathrm{QMax}(R_M)$ and is homotypic to $I$ when $I\in\mathrm{Q_iMax}(R)\cup\mathrm{Q_rMax}(R)$. If $I\in\mathrm{Q_dMax}(R)$ is of the form $I=M_1\cap M_2=M_1M_2,\ M_i\in\mathrm{Max}(R)$ for $i\in\mathbb N_2$, then $M\in\{M_1,M_2\}$. Assume, for example, that $M=M_1$, so that $M_2R_M=R_M,\ IR_M=MR _M\in\mathrm{Q_iMax}(R_M)$ and $|\mathrm{V}(I)|\leq 2$ by Theorem \ref{2.3}.  

Conversely, assume that $\mathrm{V}(I)\subseteq\mathrm{Max}(R),\ IR_M\in\mathrm{QMax}(R_M)$ for any $M\in\mathrm{V}(I)$ and $|\mathrm{V}(I)|\leq 2$ with $IR_M\in\mathrm{Max}(R_M)$ for any $M\in\mathrm{V}(I)$ when $|\mathrm{V}(I)|= 2$.

If $|\mathrm{V}(I)|=1$, then $I$ is contained in only one maximal ideal $M$ of $R$. Hence, $\sqrt I=M$ and $I$ is $M$-primary by \cite[Corollary 1, p.153]{ZS}.

If $IR_M\in\mathrm{Q_iMax}(R_M)$, then $IR_M=MR_M$ gives by \cite[Corollary 2, p.225]{ZS} that $IR_M\cap R=MR_M\cap R=M=I\in\mathrm{Q_iMax}(R)$.
  
If $IR_M\in\mathrm{Q_rMax}(R_M)$, we claim that $I$ is submaximal. Of course, $I\neq M$ because $IR_M\neq MR_M$ since the previous proof would give that $IR_M\in\mathrm {Q_iMax}(R_M)$. Assume that there exists an ideal $J$ of $R$ such that $I\subset J\subset M\ (*)$. Therefore, $J$ is also $M$-primary and $(*)$ would imply by \cite[Corollary 2, p.225]{ZS} that $IR_M\subset JR_M\subset MR_M$, a contradiction since $IR_M$ is submaximal. Then, $I\in\mathrm{Q_rMax}(R)$.

If $|\mathrm{V}(I)|=2$, then $IR_M=MR_M\in\mathrm{Max}(R_M)$ for any $M\in\mathrm{V}(I)$ and $I$ is contained in two maximal ideals $M$ and $ M'$. It follows that $I\subseteq J:=M\cap M'$, with $IR_M=JR_M,\ IR_{M'}= JR_{M'}$ and $IR_{M''}=JR_{M''}=R_{M''}$ for any $M''\in\mathrm{Max}(R)\setminus\{M,M'\}$ which implies $I=J$ and $I\in\mathrm{Q_dMax}(R)$.
\end{proof}

\begin{remark} \label{3.18} The converse part of Proposition \ref{3.161} does not hold in general if $I$ is contained in more than one maximal ideal, and never holds if $I$ is contained in at least three maximal ideals. This last result is obvious. If $I\subseteq M\cap M'$ for $M,M'\in\mathrm{Max}(R)$ such that $IR_M\in\mathrm{QMax}(R_M)$, the only possibility for $I$ to be in $\mathrm{QMax}(R)$ is that $I\in\mathrm{Q_dMax}(R)$ with $I=M\cap M'$. 
\end{remark} 

\section{Quasi-maximal ideals and  finite minimal  extensions} 

\subsection{Quasi-maximal ideals through FCP  extensions}

Before looking at minimal extensions, we consider a finite extension. A {\it chain} $\mathcal C$ of subextensions of an extension $R\subseteq S$ is a family of elements of $[R,S]$ that are pairwise comparable. We say that the extension $R\subseteq S$ has FCP (or is an FCP extension) if each chain in $[R,S]$ is finite, or equivalently, its lattice $[R,S]$ is Artinian and Noetherian. An FCP extension is finitely generated, and (module) finite if integral.  
 
If $R\subseteq S$ has FCP, any maximal (necessarily finite) chain of $[R, S]$, of the form $R=R_0\subset R_1\subset\cdots\subset R_{n-1}\subset R _n=S$, with {\it length} $n<\infty$, results from juxtaposing $n$ minimal extensions $R_i\subset R_{i+1},\ 0\leq i\leq n-1$. 
 
Minimal extensions were introduced by Ferrand-Olivier \cite{FO}. An extension $R\subset S$ is called {\it minimal} if $[R,S]=\{R,S\}$, and in this case, $\mathrm{Supp}_R(S/R)=\{M\}$ for some $M\in\mathrm{Max}(R)$ by \cite[Th\'eor\`eme 2.2]{FO}. A minimal extension is either a flat epimorphism or a strict monomorphism, in which case it is (module) finite. We recall that a {\it finite} extension $R\subset S$ is an extension such that $S$ is a finitely generated $R$-module. Three types of minimal integral extensions exist, characterized in the next theorem, (a consequence of the fundamental lemma of Ferrand-Olivier), whence there are four types of minimal extensions, mutually exclusive. From now on, if $R\subset S$ is a minimal  finite extension, we say that $R\subset S$ is in $\mathcal M$.

\begin{theorem}\label{minimal+} \cite [Theorems 2.1, 2.2 and 2.3]{DPP2}. An extension $R\subset T$ is in $\mathcal M$ if and only if $M:=(R:T)\in\mathrm{Max}(R)$ and one of the following three statements holds:

\noindent (a) {\bf inert case}: $M\in\mathrm{Max}(T)$ and $R/M\to T/M$ is a minimal field extension.

\noindent (b) {\bf decomposed case}: There exist $M_1,M_2\in\mathrm {Max}(T)$ such that $M=M_1\cap M_2$ and the natural maps $R/M\to T/M _1$ and $R/M\to T/M_2$ are both isomorphisms; or, equivalently, there exists $q\in T\setminus R$ such that $T=R[q],\ q^2-q\in M$, and $Mq\subseteq M$.

\noindent (c) {\bf ramified case}: There exists $M'\in\mathrm{Max}(T)$ such that ${M'}^2\subseteq M\subset M',\ [T/M:R/M]=2$, and the natural map $ R/M\to T/M'$ is an isomorphism, or, equivalently, there exists $q\in T\setminus R$ such that $T=R[q],\ q^2\in M$ and $Mq\subseteq M$.
 In particular, $\mathrm{c}(T/M')=\mathrm{c}(T/M)$.

In each of the above cases, $\mathrm{Supp}(T/R)=\{M\}$.
\end{theorem}

\begin{proof} We prove the last result of (c). Since $R/M\cong T/M'$, it follows that $\mathrm{c}(T/M')=\mathrm{c}(R/M)$. According to \cite[Proposition 4.7]{DPPS}, $R/M\subset T/M$ is in $\mathcal M$ and $\mathrm{c}(T/M)=\mathrm{c}(R/M)$ gives $\mathrm{c}(T/M')=\mathrm{c}(T/M)$.
\end{proof}

\begin{remark}\label{4.01}(1) The element $q\in T$ gotten in (b) (resp. (c)) of Theorem \ref{minimal+} satisfies either $q\in M_1$ or $q\in M_2$ (resp. $q\in M'$) with either $M_1=M+Rq$ or $M_2=M+Rq$ (resp. $ M'=M+Rq$). This reminds for $M$ the definition of a decomposed (resp. ramified) ideal in $T$. The analogy will be gotten in Theorem \ref{4.1}.

(2) To shorten the writing, for an extension $R\subset S$ which is in $\mathcal M$, we say that $R\subset S$ is in $\mathcal M_i$  (resp. $\mathcal M_d,\ \mathcal M_r$) if the extension is minimal inert, (resp. decomposed, ramified).

(3) Assume that $R\subset S$ is in $\mathcal M_d$ and set $M:=(R:S)$. There exist $M_1,M_2\in\mathrm{Max}(S)$ such that $M=M_1\cap M_2 $. Then, $M\prec M_1$ in $S$, with $M=M_1\cap R$. Setting $I':=M,\ J':= M_1,\ I:=I'\cap R=M$ and $J:=J'\cap R=M$, we get that $I'\prec J'$ in $S$
 and $I=J$. Then, Proposition \ref{3.35} cannot apply since $\mathrm{V}(I' \cap R )\cap\mathrm{Supp}(S/R)\neq\emptyset$.
\end{remark}
 
For any extension $R\subseteq S$, the {\it length} of $[R,S]$, denoted by $\ell[R,S]$, is the supremum of the lengths of chains of $R$-subalgebras of $S$. It should be noted that if $R\subseteq S$ has FCP, then there {\it does} exist some maximal chain of $R$-subalgebras of $S$ with length $\ell[R,S]$ \cite[Theorem 4.11]{DPP3}.

\begin{proposition}\label{3.20} A finite extension $R\subseteq S$ such that  $(R:S)\in\mathrm{QMax}(R)$  has FCP.
\end{proposition}
\begin{proof} According to Theorem \ref{2.3}, $R/(R:S)$ has at most three proper ideals, and is then an Artinian ring. Therefore $R\subseteq S$ is integral and has FCP by \cite[Theorem 4.2]{DPP2}.
\end{proof}

\begin{proposition}\label{3.1} Let $R\subseteq S$ be a finite extension, $J\in\mathrm{QMax}(S)$ and  $I:=J\cap R$.
\begin{enumerate}
\item If $J\in\mathrm{Q_iMax}(S)$, then $I\in\mathrm{Q_iMax}(R)$. 
 \item Let $J\in\mathrm{Q_dMax}(S)$, so that $J=N_1\cap N_2$ with $N_1,N_2\in\mathrm{Max}(S),$
 
 \noindent$ N_1\neq N_2$. Then: 
\begin{enumerate}
\item $N_1\cap R=N_2\cap R$ implies that $I\in\mathrm{Q_iMax}(R)\cup\mathrm{Supp}(S/R)$. 
\item $N_1\cap R\neq N_2\cap R$ implies that $I\in\mathrm{Q_dMax}(R)$. 
\end{enumerate}
\item If $J\in\mathrm{Q_rMax}(S)$, let $N:=\sqrt[S]J$. Then, $R/I$ is a local ring whose maximal ideal is  $M': =(N\cap R)/I$. 

If, in addition, $M'$ is principal, then $I\in\mathrm{Q_rMax}(R)$. 
\end{enumerate}
\end{proposition}
\begin{proof} (1) Since $R\subseteq S$ is integral, $N\cap R\in\mathrm {Max}(R)$ for any $N\in\mathrm{Max}(S)$. So (1) is obvious.

(2) Since $J\in\mathrm{Q_dMax}(S)$, then $J=N_1\cap N_2$ with $N_1, N_2\in\mathrm {Max}(S),$

\noindent $N_1\neq N_2$. Set $M_i:=N_i\cap R\in\mathrm{Max}(R)$ for $i \in\mathbb N_2$. Then, $I=N_1\cap N_2\cap R=M_1\cap M_2$.   
 
(a) If $M_1=M_2$, then $I\in\mathrm{Q_iMax}(R)\cup\mathrm{Supp}(S/R)$. This holds, for example, if $R\subset S$ is in $\mathcal{M}_d$ and $J= (R:S)$ (see Theorem \ref{minimal+}). In this case, we have $R/M_1\cong S/N_i$ for $i\in\mathbb N_2$. 
 
 (b) If $M_1\neq M_2$, then $I\in\mathrm{Q_dMax}(R)$. 

(3) Assume that $J\in\mathrm{Q_rMax}(S)$ and let $N:=\sqrt[S]J$. Then, $ S/J$ is a local ring and has finitely many ideals by Theorem \ref{2.3}. Therefore $S/J$ is an Artinian, and also a Noetherian ring. So is $R/I$ by the Eakin-Nagata Theorem because $R/I\subseteq S/J$ is finite and $R/I$ is a local ring whose maximal ideal is $M':=(N\cap R)/I$. Set $R':=R/I,\ S':= S/J$ and $N':=N/J$. According to Proposition \ref{2.5}, $S'$ is a SPIR such that $N'^2=0$. Since $M'=N'\cap R'$, we get that $M'^2=0$. 

If $M'$ is principal, that is $M'=R'x$, then $M'$ is the only ideal of $R'$ different from $0$. We deduce that $R'$ is a SPIR such that $M'^2=0$. Hence, $I\in\mathrm{Q_rMax}(R)$ by Proposition \ref{2.5}.
\end{proof}

We consider an FCP extension $R\subseteq S$ and get conditions in order that $(R:S)\in\mathrm{QMax}(S)$. We recall some needed definitions.
 
\begin{definition}\label{4.60} Let $R\subseteq S$ be an extension. The {\it seminormalization} ${}_S^+R$ of $R$ in $S$ is the greatest subintegral extension of $R$ in $S$ and the {\it $t$-closure} ${}_S^tR$ is the greatest infra-integral extension of $R$ in $S$. An extension $R\subseteq S$ is called {\it t-closed} (resp. seminormal) if ${}_S^tR=R$ (resp.${}_S^+R=R$ )                   
  (see \cite{Pic 1} and \cite{S}).
  \end{definition}
  
The case where the conductor of an extension $R\subset S$ is quasi-maximal inert  in $S$ has  a simple characterization. 

\begin{proposition}\label{4.7} If $R\subset S$ is an integral extension where $(R:S)\in\mathrm{Q_iMax}(R)$, then $(R:S)\in\mathrm{Q_iMax}(S)$ 
if and only if $R\subset S$ is   t-closed.
\end{proposition}
\begin{proof} Set $M:=(R:S)\in\mathrm{Max}(R)$. If $M\in\mathrm{Q_iMax}(S)$, then $R/M\subset S/M$ is an integral field extension. It follows that $ R/M\subset S/M$ is t-closed by \cite[Lemme 3.10]{Pic 1}, and so is $R\subset S$ by \cite[Proposition 2.13]{Pic 171}. 

Conversely, assume that $R\subset S$ is t-closed, then so is $R/M\subset S/M$, giving that $S/M$ is a field by the two previous references. So, we get that $M\in\mathrm{Q_iMax}(S)$.
\end{proof}

Most often in the rest of the paper, thanks to the previous Proposition, we will not consider the case where the conductor of an extension $R\subset S$ is in $\mathrm{Q_iMax}(S)$.
 
 The next Lemma will be useful.
 
\begin{lemma}\label{4.61} Let $R\subset S$ be an integral FCP i-extension where $(R,M)$ is a local ring, whence $S$ is also local. Let $N$ be its  maximal ideal. Set $T:={}_S^tR={}_S^+R$. 
 
Then, $\mathrm{L}_R(M/(R: S))=\mathrm{L}_S(N/(R:S))[S/N:R/M]-\ell[R,T]$.
 \end{lemma}
 
\begin{proof} Since $R\subset S$ is an integral FCP i-extension where $(R,M)$ is a local ring, $S$ is also a local ring. It follows that $T:={}_S^tR={}_S^+R$ by \cite[Proposition 5.2]{Pic 15}, and $M=N\cap R$. Set $I:=(R:S)$. Therefore, $N$ and $I$ are also ideals of $T$ and $(T,N)$ is a local ring. According to \cite[Theorem 13, p. 168]{N}, we have $\mathrm{L}_T(N/I)=\mathrm{L}_S(N/I)[S/N:T/N]$.

Since $R\subset S$ has FCP, so has $R\subset T$. Let $\{R_i\}_{i=0}^n$ be a maximal chain of $[R,T]$, with $R_0:=R,\ R_n:=T$ and $R_{i-1} \subset R_i$ in $\mathcal M$ for each $i\in\mathbb N_n$. By \cite[Lemma 5.4]{DPP2}, $n=\ell[R,T]$. Each $R_{i-1}\subset R_i$ is in $\mathcal M_r$ by \cite[Proposition 4.5]{Pic 6}. Setting $N_i:=N\cap R_i$, we get that each $(R_i,N_i)$ is a local ring having still $I$ as an ideal. Moreover, $I\subseteq N_i$. We are going to prove by induction on $i\in\mathbb N_n$ that $\mathrm{L}_{R_i}(N_i/I)=\mathrm{L}_T(N/I)-(n-i)\ (*)$. For $i=n$, we have $\mathrm{L}_{R_n}(N_n/I)=\mathrm{L}_T(N/I)$ and $(*)$ holds.

Assume that $(*)$ holds for some $i\in\mathbb N_n$, so that $\mathrm{L}_{R_i}(N_i/I)=\mathrm{L}_T(N/I)$
 
 \noindent$-(n-i)$. Since $R_{i-1}\subset R_i$ is in $\mathcal M_r$, with $ N_{i-1}=(R_{i-1}:R_i)$, we have $\mathrm{L}_{R_i}(N_i/N_{i-1})=1=\mathrm{L}_{R_i}(N_i/I)-\mathrm{L}_{R_i}(N_{i-1}/I)$. Using \cite[Lemma 5.4]{DPP2} and \cite[Theorem 13, page 168]{N}, we infer that $\mathrm{L}_{R_i}(N_i/I)-1=\mathrm{L}_{R_i}(N_{i-1}/I)=\mathrm{L}_{R_{i-1}}(N_{i-1}/I) =\mathrm{L}_T(N/I)-(n-i)-1=\mathrm{L}_T(N/I)-(n-i+1)=\mathrm{L}_T(N/I)-(n-(i-1))$ and $(*)$ holds for $i-1$. Therefore, $(*)$ holds for each $i\in\mathbb N_n$, so that $\mathrm{L}_{R}(M/I)=\mathrm{L}_T(N/I)-n=\mathrm{L}_T(N/I)-\ell[R, T ] =\mathrm{L}_S(N/I)[S/N:R/M]-\ell[R,T] $.
 \end{proof}
 
We are now able to characterize an FCP integral extension $R\subset S$ such that $(R:S)\in\mathrm{Q_dMax}(S)\cup\mathrm{Q_rMax}(S)$  
 because of Proposition \ref{4.7}.  
 
\begin{theorem}\label{4.62} Let $R\subset S$ be an FCP integral extension and set $T:={}_S^tR$. Then, 
\begin{enumerate}
\item $(R:S)\in\mathrm{Q_dMax}(S)$ if and only if either $R=T$ with $|\mathrm{V}_R((R:S))|=2$ (in this case, $(R:S)\in\mathrm{Q_dMax}(R)$) or $R\subset T$ is in $\mathcal M_d$ with $|\mathrm{V}_R((R:S))|=1$ (in this case, $(R:S)\in\mathrm{Q_iMax}(R)$). 
 
\item $(R:S)\in\mathrm{Q_rMax}(S)$ if and only if $R\subset S$ is an i-extension such that $|\mathrm{Supp}_R(S/R)|=1$, and, setting $\{M\}:=\mathrm{Supp}_R(S/R)$ and letting $N:=\sqrt[S]{MS}$, we have: $[S/N:R/M ] =\mathrm{L}_{R}(N/(R:S))$.
\end{enumerate}
\end{theorem}

\begin{proof} We set $I:=(R:S)$, an ideal shared by any element of $[R,S]$.

(1) Assume that $I\in\mathrm{Q_dMax}(S)$, that is $I=N_1\cap N_2,\ N_1 \neq N_2$, with $N_i\in\mathrm{Max}(S)$ for $i\in\mathbb N_2$. Set $M_i: =N_i\cap R$, so that $M_i\in\mathrm{Max}(R)$ for $i\in\mathbb N_2$. Then $I=M_1\cap M_2$. Two cases occur, either (a): $M_1\neq M_2$ or (b): $M_1=M_2$.  

In case (a), $I\in\mathrm{Q_dMax}(R)$ by Proposition \ref{3.1}. Moreover, $R=T$ by \cite[Proposition 2.5]{SPLIT} because $|\mathrm{V}_R(I)|=2$. 

Conversely, if $R=T$ with $|\mathrm{V}_R(I)|=2$, then, $|\mathrm{V}_R(I)|=\mathrm{V}_S(I)|=2$. According again to \cite[Proposition 2.5]{SPLIT}, $R \subset S$ is t-closed with $I$ an intersection of two maximal ideals of $S$. Then, $I\in\mathrm{Q_dMax}(R)$. 

In case (b), $I=M_1=M_2\in\mathrm{Q_iMax}(R)$ by Proposition \ref{3.1}. Moreover, $R\subset T$ is in $\mathcal M_d$ with $|\mathrm{V}_R(I)|=1$ and $|\mathrm{V}_T(I)|=|\mathrm{V}_S(I)|=2$ by Theorem \ref{minimal+}. 

Conversely, assume that $R\subset T$ is in $\mathcal M_d$ with $\mathrm {V}_R(I)|=1$. Then, $R\subset T$ and $R\subset S$ are  seminormal, which implies that $I\in\mathrm{Max}(R)$. Let $J:=(R:T)=M'_1\cap M'_2$, where $M'_i\in\mathrm{Max}(T)$ for $i\in\mathbb N_2$ and $J\in\mathrm {Max}(R)$. As $I\subseteq J$, we deduce that $I=J=M'_1\cap M'_2=N_1 \cap N_2$, where $N_i\in\mathrm{Max}(S)$ is lying over $M'_i$ for each $i\in\mathbb N_2$ (there is no harm to choose the same letters as in the first part because $N_i$ plays the same role in the two  parts). Then, 
 $I\in\mathrm{Q_dMax}(S)$. 

(2) Assume that $I\in\mathrm{Q_rMax}(S)$, and let $N:=\sqrt[S]I\in\mathrm {Max}(S)$, giving $\mathrm{L}_{S}(N/I)=1$. It follows that $\mathrm{Supp}_R(S/R)=\{M\}$ where $I\subseteq M:=N\cap R\in\mathrm{Max}(R)$ by \cite[Proposition 5.7 and Example 5.15]{Pic 18}. This implies that $R\subset S$ is an i-extension and $N=\sqrt[S]{MS}\in\mathrm{Max}(S)$ is the only maximal ideal of $S$ lying above $M$. According to Lemma \ref{4.61} because we may assume that $(R,M)$ is a local ring by \cite[Theorem 12, p. 166]{N}, $\mathrm{L}_R(M/I)=\mathrm{L}_S(N/I)[S/N: R/M]-\ell[R,T]= [S/N:R/M]-\ell[R,T]=[S/N:R/M]-\mathrm{L}_R(N/M)\ (*)$ by \cite[Lemma 5.4]{DPP2} because $R\subseteq T$ is subintegral and $N\in\mathrm{Max}(T)$ since $I\subseteq(T:S)$. Then $(*)$ implies that $\mathrm{L}_R(M/I)+\mathrm{L}_R(N/M)=\mathrm{L}_R(N/I)=[S/N:R/M]$.

Conversely, assume that $R\subset S$ is an i-extension with $|\mathrm {Supp}_R(S/R)|\\=1$, and, setting $\{M\}:=\mathrm{Supp}_R(S/R)$, with $I\subseteq M$ by \cite[Proposition 5.7 and Example 5.15]{Pic 18}, and letting $N:=\sqrt[S]{MS}\in\mathrm{Max}(S)\cap\mathrm{Max}(T)$, the following equality holds: $[S/N:R/M]=\mathrm{L}_{R}(N/I)$. Thanks to Lemma \ref{4.61}, \cite[Theorem 13, p. 168]{N} and because $\mathrm {Supp}_R(S/R)=\{M\}$, we deduce that $\mathrm{L}_R(M/I)=\mathrm{L}_{R_M}(MR_M/IR_M)=\mathrm{L}_{S_M}(NR_M/IR_M)$

\noindent$[S/N:R/M]-\ell[R_M,T_M]=\mathrm{L}_S (N/I)[S/N:R/M]-\ell[R,T]=\mathrm{L}_S(N/I)\\ \mathrm{L}_R(N/I)-\ell[R,T]=\mathrm{L}_S(N/I)\mathrm {L}_R(N/I)-\mathrm{L}_R(N/M)$ using again \cite[Lemma 5.4]{DPP2}, which gives $\mathrm{L}_S(N/I)\mathrm{L}_R(N/I)=\mathrm{L}_R(M/I)+\mathrm{L}_R(N/M)=\mathrm{L}_R(N/I)$, whence $\mathrm{L}_S(N/I)=1$, because $N\neq I$. Then $N\succ I$ and $I\in\mathrm{Q_rMax}(S)$.
\end{proof} 

\subsection{Quasi-maximal ideals in minimal  extensions}

As we recall in Theorem \ref{minimal+} and Remark \ref{4.01}, the conductor of a minimal integral extension $R\subset S$ has the same characterization as a quasi-maximal ideal of $S$. In this section, we consider the behavior of quasi-maximal ideals of $S$ linked to the conductor of a minimal integral extension $R\subset S$. In fact, the following Theorem shows that the conductor of a minimal integral extension $R\subset S$ is in $\mathrm{QMax}(S)$.

\begin{theorem}\label{4.1} Let $R\subset S$ be in $\mathcal M$. Then $(R:S)\in\mathrm{QMax}(S)$ and is of the same type as   $R\subset S$. 
\end{theorem}
\begin{proof} Set $M:=(R:S)\in\mathrm{Max}(R)$. 

By Theorem \ref{minimal+}, in the inert case, we have $M\in\mathrm{Max}(S)$ and $M\in\mathrm{Q_iMax}(S)$. 

In the decomposed case, there exist $M_1,M_2\in\mathrm{Max}(S)$ such that $M=M_1\cap M_2$ and $M\in\mathrm{Q_dMax}(S)$. 

In the ramified case, there exists $M'\in\mathrm{Max}(S)$ such that $M'^2 \subseteq M\subset M'$ and there exists $q\in S\setminus R$ such that $ M'=M+Rq$ according to Theorem \ref{minimal+}. In particular, $M$ is $M' $-primary in $S$. Then, \cite[Corollary 2, p.237]{ZS} shows that $M'\succ M$ in $S$, whence $M\in\mathrm{Q_rMax}(S)$ by Theorem \ref{2.6} and Definition \ref{2.2}.
\end{proof}

The next Proposition gives a converse of Theorem \ref{4.1} under a special assumption.

\begin{proposition}\label{4.6} Let $R\subset S$ be an FCP infra-integral extension such that $(R:S)\in\mathrm{QMax}(S)$. Then, $R\subset S$ is a minimal extension of the same type as $(R:S)$ (in $S$), so that $(R:S)\in\mathrm{Max}(R)$.
\end{proposition}
\begin{proof} Set $I:=(R:S)$. We consider the two cases of Theorem \ref{4.62} because $I\not\in\mathrm{Q_iMax}(S)$ since $S={}_S^tR$ by Proposition \ref{4.7}. 

If $I\in\mathrm{Q_dMax}(S)$, then $R\subset S$ is in $\mathcal M_d$ and $I\in\mathrm {Max}(R)$ because $R\neq{}_S^tR$. 

If $I\in\mathrm{Q_rMax}(S)$, then $R\subset S$ is an i-extension, giving that $R\subset S$ is subintegral, because infra-integral and $|\mathrm {Supp}_R(S/R)|=1$. Set $\{M\}:=\mathrm{Supp}_R(S/R)$ with $I\subseteq M$ and let $N:=\sqrt[S]{MS}$. The following equality holds: $[S/N:R/M]=\mathrm{L}_{R}(N/I)$. But $S/N\cong R/M$ because $R\subset S$ is infra-integral, so that $\mathrm{L}_{R}(N/I)=\mathrm{L}_{R}(N/M)+\mathrm{L}_ {R}(M/I)=1$. As $N\neq M$ since $S\neq R$, it follows that $\mathrm{L}_ {R}(N/M)\neq 0$ and $\mathrm{L}_{R}(M/I)=0$. Therefore, $I=M$ and Theorem \ref{minimal+} gives that $R\subset S$ is in $\mathcal M_r$. In fact, $[S/M: R/M]=2$ because $N=M+Sx$ for some $x\in N\setminus M$ and the isomorphism $R/M\cong S/N$ shows that $S=R+N=R+Rx+Nx=R+Rx$, giving $[S/M:R/M]=2$.
\end{proof}

One can ask whether Theorem \ref{4.1} has a converse. In fact, the next 
 Theorem shows that, given a ring $S$ and $I\in\mathrm{QMax}(S)$, we have to add additional assumptions in order that there exists a subring $R$ of $S$ such that $R\subset S$ is an extension in $\mathcal M$ with $(R:S)=I$. 

\begin{theorem}\label{4.15} Let $S$ be a ring and $I\in\mathrm{QMax}(S)$. Then there exists a subring $R$ of $S$ such that $I=(R:S)$ with $R\subset S$ in $\mathcal M$ if and only if one of the following conditions holds:
\begin{enumerate}
\item $I\in\mathrm{Q_iMax}(S)$ and $S/I$ is a field which has a maximal subring which is a field.
\item $I\in\mathrm{Q_dMax}(S)$ is such that $I=M_1\cap M_2$ where $ M_1,M_2\in\mathrm{Max}(S),\ M_1\neq M_2$ and $S/M_1\cong S/M_2$.
 \item $I\in\mathrm{Q_rMax}(S)$ is such that $M^2\subseteq I\subset M$, where $M\in\mathrm{Max}(S)$  and  $\mathrm{c}(S/M)=\mathrm{c}(S/I)$.
\end{enumerate}
In each of these situations, $R\subset S$ is a minimal extension of the same type as  the quasi-maximal ideal $I$. 
\end{theorem}
\begin{proof} Let $S$ be a ring and $I\in\mathrm{QMax}(S)$. Assume that there exists a subring $R$ of $S$ such that $I=(R:S)$ with $R\subset S$ in $\mathcal M$. Then Theorems \ref{4.1} and \ref{minimal+} give the first part of the equivalence. 

Now, assume that one of conditions (1), (2) or (3) holds. According to \cite[Proposition 4.7]{DPPS}, we get that there exists a subring $R$ of $S$ such that $R\subset S$ is in $\mathcal M$ with $I=(R:S)$ if and only if $R/I\subset S/I$ is in $\mathcal M$, and if this statement holds, the two extensions have the same type. In this case, since $I=(R:S)$ with $R\subset S$ an extension in $\mathcal M$, it follows that $I\in\mathrm{Max}(R)$, so that $R/I$ is a field. Therefore, if there exists a subring $K$ of $S$ which is a field such that $K\subset S/I$ is an extension in $\mathcal M$, the wanted ring $R$ is obtained by the  pullback  
$\begin{matrix}
        R         & \subset &         S            \\
\downarrow &      {}     & p\downarrow   \\
       K          & \subset &       S/I         
\end{matrix}$  
where $p:S\to S/I$ is the natural  surjective map. Then, $R=p^{-1}(K)$. For such a ring $R$, we get that $I$ is an ideal shared by $R$ and $S$, with $ K=R/I$, so that $I=(R:S)\in\mathrm{Max}(R)$ and $R\subset S$ is an  
extension in $\mathcal M$ of the same type as $K\subset S/I$ by the previous reference. So, our problem reduces to find conditions in order that there exists a maximal subring $K$ of $S/I$ which is a field and such that $K\subset S/I$ is an extension in $\mathcal M$ of the same type as the quasi-maximal ideal $I$ in $S$. We make a discussion considering the different types of the quasi-maximal ideal $I$ of $S$.

Assume that $I\in\mathrm{Q_iMax}(S)$ so that $S/I$ is a field and assume that $S/I$ has a maximal subring $K$ which is a field. Then $K\subset S/I$ is in $\mathcal M_i$, and so is $R:=p^{-1}(K)\subset S$. In Remark \ref{4.14}, we give a more precise characterization of this situation.

Assume that $I\in\mathrm{Q_dMax}(S)$, that is $I=M_1\cap M_2$ where $ M_1,M_2\in\mathrm{Max}(S),\ M_1\neq M_2$. Then, $S/I\cong S/M_1 \times S/M_2$. Assume moreover that $S/M_1\cong S/M_2$. Setting $K:= S/M_1\cong S/M_2$, we get that $K\subset S/I\cong K^2$ is in $\mathcal M_d$, and so is $R:=p^{-1}(K)\subset S$. 

Assume that $I\in\mathrm{Q_rMax}(S)$, that is $M^2\subseteq I\subset M$ for some $M\in\mathrm{Max}(S)$ and assume that $\mathrm{c}(S/M)=\mathrm{c}(S/I)$. Since $I$ is ramified in $S$, it follows that $S':=S/I$ is a SPIR by Proposition \ref{2.5} whose maximal ideal is $M':=M/I$ nilpotent of index 2. In particular, $S'$ is a complete local ring with $M'$ is principal by \cite[Definition 9]{Hun}. Set $M'=S't$. According to \cite[Theorem 8]{Hun}, $S'=f(K[[X]])$ where $f:K[[X]]\to S'$ is a morphism such that $f(X)=t$, with $K=S'/M'\cong S/M$ because $\mathrm{c}(K)=\mathrm{c}(S/M)=\mathrm {c}(S/I)$. Then, $S'=K[t]=K+Kt$, because $t^2=0$. Therefore $K\subset S' =S/I$ is in $\mathcal M_r$ by Theorem \ref{minimal+}, and so is $R:=p^{-1}(K)\subset S$. 
\end{proof}

\begin{remark}\label{4.14}The condition occurring in case $I\in\mathrm{Q_i Max}(S)$ is drastic. In fact, for a field $K$, Azarang and Karamzadeh get that $K$ has no maximal subring if and only if there is a prime number $p$ such that either $K=\mathbb F_p:=\mathbb Z/p\mathbb Z$ or $K$ is an infinite subfield of the algebraic closure $\overline{\mathbb F_p}$ of $\mathbb F_p$ such that $K=\cup_{n\in T}\mathbb F_{p^n}$, where $T\subseteq\mathbb N$ is such that $\cup_{n\in T}\mathbb F_{p^n}$ is a subfield of $\overline{\mathbb F_p}$, where $\mathbb F_{p^n}$ is the unique subfield of $\overline{\mathbb F_p}$ with $p^n$ elements ($T$ is called a FG-set) and $T$ has no maximal FG-subset \cite[Theorem 1.8] {AK}. Then, taking the contraposition, we get a characterization of a field having a maximal subring. But we have now to exclude the solutions which are not fields. Now, use \cite[Theorem 3.3]{AZ} which says that for a field $ K$ which has a maximal subring $R$, then either $R$ is a field such that $[K:R]<\infty$ or $R$ is a G-domain which is not a field.  
\end{remark}

Let $R\subseteq S$ be an extension in $\mathcal M$ and set $M:=(R:S)$. By Theorem \ref{minimal+}, $\mathrm{Supp}(S/R)=\{M\}$. According to Corollary \ref{3.3}, if $I\in\mathrm{QMax}(R)$ is such that $\mathrm{V}(I)\cap\mathrm{Supp}(S/R)=\emptyset$, then $IS\in\mathrm{QMax}(S)$ and is homotypic to $I$, with $I=IS\cap R$. Moreover, if $I$ is a submaximal ideal of $R$, then $IS\prec N':=NS\in\mathrm{Max}(S)$ for any $N\in\mathrm {V}(I)$. The following Proposition gives a converse for this situation. 

\begin{proposition}\label{4.11} Let $R\subset S$ be in $\mathcal M$, $M:= (R:S)$ and $J\in\mathrm{QMax}(S)$ be such that $I:=J\cap R\nsubseteq M$. Then $J=IS,\ I\in\mathrm{QMax}(R)$ and is homotypic to $J$. 
\end{proposition}
\begin{proof} We begin to notice that $M\in\mathrm{Max}(R)$ and $\mathrm {Supp}(S/R)=\{M\}$.
 
If $J\in\mathrm{Q_iMax}(S)$, then $I\in\mathrm{Q_iMax}(R)$ by Proposition \ref{3.1}. But, $I\subseteq J$ implies $IS\subseteq J$ with $IS\in\mathrm{Max}(S)$ since $I\not\in\mathrm{Supp}(S/R)$ by Lemma \ref{3.46}. Hence, $J=IS$. 
 
Assume now that $J\in\mathrm{Q_dMax}(S)\cup\mathrm{Q_rMax}(S)$. Then, $J\prec N'$ for some $N'\in\mathrm{QMax}(S)$. According to Proposition \ref{3.35}, $I\prec M':=N'\cap R\in\mathrm{Max}(R)$ and $IS=J$. In particular, $I\in\mathrm{QMax}(R)$. By Corollary \ref{3.3}, $IS$ is homotypic to $I$. 
\end{proof}

Now, according to Proposition \ref{4.11} and considering an extension $R\subset S$ in $\mathcal M$ with $M:=(R:S)$, in order to study the behavior of ideals $I\in\mathrm{QMax}(R)$ or $J\in\mathrm{QMax}(S)$ such that $I =J\cap R$, it remains to assume the situation where $I\subseteq M$. The following key Lemma characterizes submaximal ideals $I$ of $R$ which are also ideals of $S$, that is $I=IS$.

\begin{lemma}\label{4.3} Let $R\subset S$ be an  extension such that $(R: S)\in\mathrm{Max}(R)$ and let $I$ be a submaximal ideal of $R$. Then $I= IS$ if and only if $I\subset(R:S)$ and there exists an ideal $J$ of $S$ such that $I=J\cap R$. 
\end{lemma}
\begin{proof} Set $M:=(R:S)$. 

Assume that $I\subset M$ and $J$ is an ideal of $S$ such that $I=J\cap R$. This implies $I\subseteq J$, so that $IS\subseteq JS=J\ (*)$. We infer that $I\subseteq IS\subseteq MS=M$. Then $IS\in\{I,M\}$ because $IS$ is also an ideal of $R$. We claim that $IS=I$. Otherwise, $IS=M$ and $(*)$ would imply $M\subseteq J$, so that $M=M\cap R\subseteq J\cap R=I$, a contradiction and  $IS=I$ holds. 
 
Conversely, if $I$ is a submaximal ideal of $R$ such that $I=IS$, then $I\subset M$ because $I$ is also an ideal of $S$ and $I\not\in\mathrm{Max}(R)$. Moreover, $IS$ is also an ideal of $S$ such that $I=IS\cap R$.
\end{proof}

\begin{remark}\label{4.4} (1) Let $R\subset S$ in $\mathcal M_d$ and set $ M:=(R:S)=M_1\cap M_2\in\mathrm{Max}(R)\setminus\mathrm{Max}(S)$ with $M_i\in\mathrm{Max}(S)$ for $i\in\mathbb N_2$ and $ M_1\neq M_2 $. Let $M'\in\mathrm{Max}(R),\ M'\neq M$ and set $I:=M\cap M'=MM'$. Then $I\prec M'$ in $R$. Moreover, $MS=M$ and $IS=(MM')S=MSM'=MM' =I$ is also an ideal of $S$. Since $M'\neq M$, we have $N':=M'S\in\mathrm{Max}(S)$ thanks to Lemma \ref{3.46} and $I\subset M_1\cap N' \subset N'$ shows that $IS\not\prec M'S$. Then, Proposition \ref{3.35}(1) cannot apply since $\mathrm{V}(I)\cap\mathrm{Supp}(S/R)\neq\emptyset$.

(2) In order to characterize extensions $R\subset S$ in $\mathcal M$ such that, given an ideal $I\in\mathrm{QMax}(R)$ with $I\subset(R:S)$, there exists $J\in\mathrm{QMax}(S)$ such that $I=J\cap R$, we can add the assumption that $IS=I$ by Lemma \ref{4.3}. This means that $R$ and $S$ shares $I$.
\end{remark}

In the following subsections, we are going to give more precise results concerning quasi-maximal ideals of $R$ and $S$ according to the type of the minimal extension $R\subset S$.

\subsection{Quasi-maximal ideals in minimal inert  extensions}\begin{proposition}\label{4.5} Let $R\subset S$ be in $\mathcal M$, $M:= (R:S)$ and $I\in\mathrm{Q_dMax}(R)$ with $I\subset M$. Then  $I=M\cap M'$, for some $M'\in\mathrm{Max}(R)\setminus\{M\}$ and there exists $J:= N\cap N'\in\mathrm{Q_dMax}(S)$ such that $I=J\cap R$, where $N$ (resp. $N'$) $\in\mathrm{Max}(S)$ is lying in $S$ above $M$ (resp. $M'$) and $I =IS$. Moreover, $I\in\mathrm{Q_dMax}(S)$ if and only if $R\subset S$ is in $\mathcal M_i$. In this case, $I$ is the only ideal of $\mathrm{QMax}(S)$ contained in $M$ and lying above $I$. If $M\not\in\mathrm{Max}(S)$, then $I\neq J$. 
\end{proposition}
\begin{proof} 

According to Theorem \ref{4.1}, $M\in\mathrm{QMax}(S)$. In the proof of Proposition \ref{3.4}, we proved that there exists $J:= N\cap N'\in\mathrm {Q_dMax}(S)$ such that $I=J\cap R$, where $N$ (resp. $N'$) $\in\mathrm {Max}(S)$ is lying above $M$ (resp. $M'$). By Lemma \ref{4.3}, $I=IS$ holds.  

If $M\in\mathrm{Max}(S)$, then $R\subset S$ is in $\mathcal M_i$
and $I=IS$ is also in $\mathrm{QMax}(S)$ because any ideal of $S$ contained between $I$ and $M$ is an ideal of $R$, and then is either $M$ or $I$. In this case, $M=N$, so that $J=N\cap N'=M\cap N'=M\cap R\cap N' =M\cap M'=I$. But $N'=M'S\in\mathrm{Max}(S)$ is the unique element of $\mathrm{Max}(S)$ lying over $M'$ because $M'\not\in\mathrm{Supp}(S/R)$, according to Lemma \ref{3.46}. Then, $I\in\mathrm{Q_dMax}(S)$ 
 because $M\neq N'$.

If $M\not\in\mathrm{Max}(S)$, then $R\subset S$ is not in $\mathcal M_i$
and $I\neq J$. Indeed, $I_M=M_M\neq J_M=N_M$, so that $I\subset J$. In particular, $I\not\in\mathrm{Q_dMax}(S)$. 
\end{proof} 

\begin{proposition}\label{4.81} Let $R\subset S$ be an extension in $\mathcal M$ and $J$ be a submaximal ideal of $S$ with $J\subset (R:S)$. 
Then either $J\in\mathrm{Q_dMax}(R)\cap\mathrm{Q_dMax}(S)$ or $J\in\mathrm{Q_rMax}(S)\setminus\mathrm{QMax}(R)$, with $R\subset S$ in $\mathcal M_i$ in both  cases. 
\end{proposition}
\begin{proof} $J\not\in\mathrm{Q_iMax}(S)$ as a submaximal ideal of $S$. 

If $J\in\mathrm{Q_dMax}(S)$, then $J=N_1\cap N_ 2,\ N_1\neq N_2,\ (*)$, with $N_i\in\mathrm{Max}(S)$ for $i\in\mathbb N_2$. Setting $M_i:=N_i\cap R$ for $i\in\mathbb N_2$, we get that $J=M_1\cap M_2$ because $J\subset(R:S)\subset R$. Therefore $(R:S)=M_i$ for some $i\in\mathbb N_2 $. Assume, for example, that $(R:S)=M_1$. If $M_1=M_2$, it follows that $ J=M_1\in\mathrm{Max}(R)$, a contradiction. Then, $J= M_1\cap M_2,\ M_1\neq M_2\ (**)$ and $J\in\mathrm{Q_dMax}(R)$. Moreover, $(R:S)\not\subseteq N_2 $. Otherwise, $J=(R:S)$, a contradiction. Then, $R\subset S$ is not in $\mathcal M_d$. Localizing $(*)$ and $(**)$ at $M_1$, we get $J_{M_1}={M_1}_{M_1}={N_1}_{M_1}$, and, because $M_2\not\in\mathrm {Supp}(S/R)$, we have ${M_1}_{M_2}=R_{M_2}=S_{M_2}={N_1}_{M_2}$ since $R_{M_2}=S_{M_2}$, so that $M_1=N_1=(R:S)$ and $R\subset S$ is in $\mathcal M_i$. 

If $J\in\mathrm{Q_rMax}(S),$ then $N^2\subseteq J\subset N\ (***)$ for some $N\in\mathrm{Max}(S)$ and $J$ is an $N$-primary ideal of $S$. Set $M:=N\cap R\in\mathrm{Max}(R)$. Then $(***)$ implies $M^2\subseteq N^2\subseteq J\subseteq M\subseteq N$, with $J$ and $N$ ideals of $S$ and $J$ and $M$ ideals of $R$ such that $J$ is an $M$-primary ideal of $ R$. Moreover, $M=(R:S)$ because $M$ and $(R:S)$ are both maximal ideals of $R$ which contain $J$. As $J$ is a submaximal ideal of $S$ contained in only one maximal ideal of $S$, that is $N$ by $(***)$, we have $M=(R:S)=N\cap R$. But $M= (R:S)$ shows that $M$ is also an ideal of $ S$ and $J\subseteq M\subseteq N$ shows that either $M=J$ or $M=N$. Now, $N$ (resp. $M$) is the only maximal ideal of $S$ (resp. $R$) containing $J$. Then, the equality $M=J$ would imply $J=(R:S)$, a contradiction. To conclude, we have $M=N=(R:S)$ and $R\subset S$ is in $\mathcal M_i$. Since $M\succ J$ in $S$, this gives that $\mathrm{L}_S (M/J)=1$. But, by \cite[Theorem 13, p.168]{N}, $\mathrm{L}_R(M/J)=\mathrm{L}_S(M/J)[S/M:R/M]=[S/M:R/M]> 1$ because $R/M\subset S/M$ is a field extension which is not an isomorphism. Then, $J\not\in\mathrm {QMax}(R)$.
\end{proof}

\begin{corollary}\label{4.8} Let $R\subset S$ be in $\mathcal M_i$ and $I\in\mathrm{Q_rMax}(R)$ such that $I\subset (R:S)$. Then there does not exist $J\in\mathrm{QMax}(S)$ such that $J\subset(R:S)$ and $I=J\cap R$. 
\end{corollary}
\begin{proof} Since $R\subset S$ is in $\mathcal M_i$, $(R:S)\in\mathrm {Max}(S)$. Assume that there exists $J\in\mathrm{QMax}(S),$ $J\subset (R:S)$ such that $I=J\cap R$, so that $I=J$. According to Proposition \ref{4.81}, if $J\in\mathrm{Q_rMax}(S)$, we get that $J=I\not\in\mathrm {QMax}(R)$, a contradiction with the hypothesis on $I$. The same Proposition says that if $J\in\mathrm{Q_dMax}(S)$, then $I=J\in\mathrm {Q_dMax}(R),$ a contradiction. At last, if $J\in\mathrm{Q_iMax}(S)$, then $I=J\in\mathrm{Q_iMax}(R)$, also  a contradiction since $I\subset (R:S)$.
 
It follows that there does not exist $J\in\mathrm{QMax}(S)$ such that $J\subset (R:S)$ and  $I= J\cap R$.
\end{proof}

\begin{remark}\label{4.9} Let $R\subset S$ be in $\mathcal M_i$ and $J\in\mathrm{Q_rMax}(S)$ such that $J\subset (R:S)$. In Proposition \ref{4.81}, we prove that $J\not\in\mathrm{QMax}(R)$ as an ideal of $R$. In particular, setting $M:=(R:S)\in\mathrm{Max}(S)$, there exists an ideal $K$ of $R$ such that $J\subset K\subset M$. Therefore $J= JS\subset K\subseteq KS\subseteq MS=M$ in $S$. So, we have $J\subset KS\subseteq M$, with $KS$ an ideal of $S$. But $M\succ J$ in $S$, so that $ KS=M$. Then, any ideal of $R$ strictly contained between $J$ and $M$ is lifted up in $M$. 

This situation may happen for example, in the context of algebraic numbers. Take an extension of algebraic orders $R\subset S$ in $\mathcal M_i$ and set $M:=(R:S)\in\mathrm{Max}(S)$, so that $M^2\neq M$ by the Nakayama Lemma. As $0<\mathrm{L}_S(M/M^2)<\infty$, there exists an ideal $J$ of $S$ such that $M^2\subseteq J\subset M$ with $\mathrm{L}_ S(M/J)=1$. Then, $J\in\mathrm{Q_rMax}(S)$ with $J\subset M$ and $J\not\in\mathrm {QMax}(R)$. 
\end{remark}

Let $R\subset S$ be an extension with $I\in\mathrm{QMax}(R)$ such that $\mathrm{V}_R(I)\cap\mathrm{Supp}(S/R)=\emptyset$. By Corollary \ref{3.3}, $I=IS\cap R$ and $IS\in\mathrm{QMax}(S)$ with $IS\subset N:= MS\in\mathrm{Max}(S)$ for any $M\in\mathrm{V}_R(I)$ and $IS$ is homotypic to $I$. Now, we are going to consider a minimal extension $R\subset S$ and an ideal $J\in\mathrm{QMax}(S)$ such that $J\not\subseteq(R:S)$. In particular, Proposition \ref{4.82} generalizes Proposition \ref{4.11}. According to Proposition \ref{3.1}, it is enough to  consider only $J\in\mathrm{Q_rMax}(S)$.

\begin{proposition}\label{4.82} Let $R\subset S$ be in $\mathcal M_i$ and $J\in\mathrm{Q_rMax}(S),\ J\not\subseteq (R:S)$. Set $I:=J\cap R$. Then $I\in\mathrm{Q_rMax}(R)$ and $J=IS$.
\end{proposition} 
\begin{proof} Set $M:=(R:S)$. Since $R\subset S$ is in $\mathcal M_i$, then $\mathrm{Supp}(S/R)=\{M\}$ and $M\in\mathrm{Max}(S)$. It follows that $J\not\subseteq M$. Let $N':=\sqrt[S]J\in\mathrm{Max}(S)$ and set $ N:=N'\cap R\in\mathrm{Max}(R)$. Then $N\not\in\mathrm{Supp}(S/R)$ and $\{N\}=\mathrm{V}_R(I)$. According to Proposition \ref{4.11}, $I\in \mathrm{Q_rMax}(R)$ and $J=IS$ because $J\not\subseteq M$.
\end{proof}

 \subsection{Quasi-maximal ideals in minimal decomposed   extensions}
 
Let $R\subset S$ be in $\mathcal M$. According to Proposition \ref{4.81}, if $J\in\mathrm{QMax}(S)$ is such that $J\subset (R:S)$, then $R\subset S$ is in $\mathcal M_i$. It follows that we cannot consider this situation for $R\subset S$ in $\mathcal M_d\cup\mathcal M_r$. 

If we delete the assumption that $J\subset(R:S)$, but keep $I:=J\cap R\subset (R:S)$, we have more results as we can see in the next Proposition. In the case $I\not\subseteq(R:S)$, then $I\in\mathrm{QMax}(R)$ is homotypic to $J$ and $J=IS$ by Proposition \ref{4.11}. 

\begin{proposition}\label{4.111} Let $R\subset S$ be in $\mathcal M_d$ with $M:=(R:S)=M_1\cap M_2,\ M_i\in\mathrm{Max}(S)$ for $i\in\mathbb N _2$. Let $I\in\mathrm{QMax}(R)$ be such that $I\subset M$. There exists $J\in\mathrm{QMax}(S),\ J\not\subseteq M$ such that $I=J\cap R$ if and only if $I=IS$ and either (1) $I\in\mathrm{Q_dMax}(R)$ with $I=M\cap N,\ N\in\mathrm{Max}(R),\ N\neq M$ or (2) $I$ is an $M$-primary ideal of $\mathrm{Q_rMax}(R)$.

If these conditions are satisfied, then $I$ and $J$ are homotypic and $J=M _i\cap N'$ for some $i\in\mathbb N_2$, where $N':=NS\in\mathrm{Max}(S)$ in case (1) or $\mathrm{V}_S(J)=\{M_i\}$ for only one $i\in\mathbb N_2$ in case (2). 
\end{proposition}

\begin{proof} First, $I\not\in\mathrm{Q_iMax}(R)$ because $I\subset M\in\mathrm{Max}(R)$. Next, $M\prec M_i$ in $S$ for any $i\in\mathbb N_2$.

Assume that there exists $J\in\mathrm{QMax}(S)$ with $J\not\subseteq M$ such that $I=J\cap R$. Then $J\not\in\mathrm{Q_iMax}(S)$ by Proposition \ref{3.1} and $I=IS$ by Lemma \ref{4.3}.  

(1) If $I\in\mathrm{Q_dMax}(R)$, then Proposition \ref{3.1} gives that $J\in\mathrm{Q_dMax}(S)$. Since $I\subset M$, we get that $I=M\cap N,\ N\in\mathrm{Max}(R),\ N\neq M$. It follows that $N\not\in\mathrm{Supp}(S/R)$ and thanks to Lemma \ref{3.46}, $N':=NS\in\mathrm{Max}(S)$ satisfies $N =N'\cap R$, with $N'\neq M_i$ for $i\in\mathbb N_2$. Therefore $I=M_1 \cap M_2\cap N=M_1\cap M_2\cap N'\subset J$ in $S$. As $J\in\mathrm {Q_dMax}(S)$, we obtain $J=M_i\cap N'$ for some $i\in\mathbb N_2$ because $J=M_1 \cap M _2$ would imply $J=M$, a contradiction.

(2) If $I\in\mathrm{Q_rMax}(R)$, then Proposition \ref{3.1} gives that $J\in\mathrm{Q_rMax}(S)$. Moreover, $I\subset M$ and $I$ is a primary ideal 
 of $R$ because ramified. This implies that $I$ is $M$-primary and satisfies $M_1^2M_2^2=M^2\subseteq I\subset J\ (*)$. But, $J\in\mathrm {Q_rMax}(S)$ is a primary ideal of $S$ and then an $M_i$-primary ideal. Then, $J\subseteq M_i$ for only one $i\in\mathbb N_2$.

Conversely, assume that $I=IS$, so that $I$ is also an ideal of $S$, and
either (1) $I\in\mathrm{Q_dMax}(R)$ with $I=M\cap N,\ N\in\mathrm{Max}(R),\ N\neq M$ or (2) $I$ is an $M$-primary ideal of $\mathrm{Q_rMax}(R)$. We claim that there exists $J\in\mathrm{QMax}(S),\ J\not\subseteq M$
 such that $I=J\cap R$.
 
(1) Assume that $I\in\mathrm{Q_dMax}(R)$. We keep the first part of the proof with $I=M\cap N,\ N\in\mathrm{Max}(R),\ N\neq M,\ N':=NS\in\mathrm{Max}(S)$ which satisfies $N=N'\cap R$, with $N'\neq M_i$ for $i\in\mathbb N_2$. Setting $J:=M_1\cap N'$, we get that $J\in\mathrm{Q_d Max}(S)$ and $J\cap R=M_1\cap N'\cap R=M\cap N=I$. Moreover, $J\not\subseteq M$. Otherwise, we should have $M_1\cap N'\subseteq M\subset M_2$, a contradiction. A similar result holds taking $J':=M_2\cap N'$.

(2) Assume that $I$ is an $M$-primary ideal of $\mathrm{Q_rMax}(R)$. According to Lemma \ref{4.3}, there exists an ideal $J'$ of $ S$ such that $ I=J'\cap R$. We claim that we can choose $J'\in\mathrm{Q_rMax}(S)$. Since $I\prec M$, we get that $\mathrm{L}_R(M/I)=1$. By \cite[Theorem 13, p. 168]{N}, we also have $\mathrm{L}_R(M/I)=\sum_{i=1}^2\mathrm{L}_{S_{M_i}}[(M/I)_{M_i}][S/M_i:R/M]=\sum_{i=1}^2\mathrm{L}_{S_{M_i}}[(M/I)_{M_i}]$ because $R\subset S$ is infra-integral. It follows that $\mathrm {L}_{S_{M_i}}[(M/I)_{M_i}]=1$ for some $i$ and $\mathrm{L}_{S_{M_j}}[(M/I)_{M_j}]=0$ for $j\neq i$. In particular, $I_{M_j}=M_{M_j}$ and $I_{M_i}\prec M_{M_i}$. Using \cite[Corollary 2, p. 225]{ZS},  there exists an ideal $ M_i$-primary $J$ of $S$ such that $J_{M_i}=I_{M_i}$, giving $\mathrm{L}_ S[(M_i/J)]=1\ (*)$, and $J\in\mathrm{Q_rMax}(S)$ holds with $J\prec M_i$ in $S$. Moreover, $I_{M_j}\subset J_{M_j}=S_{M_ j}\ (**)$ shows that $I\subseteq J\cap R\subseteq M_i\cap R=M=M_1\cap M_2$. It follows that $J\not\subseteq M$, otherwise, we should have $J\subseteq M_j$, a contradiction with $(**)$. At last, since $I\prec M$ in $R$, we get that $J\cap R\in\{I,M\}$. But $J\cap R=M$ gives $M\subseteq J$, in contradiction with $(*)$ because $M\prec M_i$. To conclude, $J\cap R=I$. 
 \end{proof}
 
\begin{proposition}\label{4.12} Let $R\subset S$ be in $\mathcal M_d$ with $M:=(R:S)=M_1\cap M_2,\ M_i\in\mathrm{Max}(S)$ for $i\in\mathbb N_2$. Let $J\in\mathrm{QMax}(S),\ J\not\subseteq M$. Set $I:=J\cap R$. Then $I\in\mathrm{QMax}(R)$ and is homotypic to $J$.
\end{proposition}

\begin{proof} If $J\in\mathrm{Q_iMax}(S)$, then $I =J\cap R\in\mathrm {Max}(R)$. 

If $J\in\mathrm{Q_dMax}(S)$, then by Proposition \ref{3.1}, $I\in\mathrm {Q_iMax}(R)\cup\mathrm{Q_dMax} (R)$. We claim that $I\in\mathrm {Q_dMax}(R)$. Otherwise $I\in\mathrm{Q_iMax}(R)$, implies $I\in\mathrm {Max}(R)$ and $I=(R:S)=M=J$, a contradiction. Therefore $I\in\mathrm {Q_dMax}(R)$. 
 
Let $J\in\mathrm{Q_rMax}(S),\ J\not\subseteq M$, so that $J$ is a primary ideal of $S$ contained in only one maximal ideal of $S$.  

Assume first that $J\subset M_i$, for some $i\in\mathbb N_2$, and for example, $i=1$, so that $M_1^2\subseteq J\subset M_1$ and $J$ is $M_1 $-primary. Then, $I=J\cap R=J\cap M_1\cap R=J\cap M=J\cap M_1\cap M _2=J\cap M_2\ (*)$, so that $I$ is also an ideal of $S$. According to \cite[Theorem 13, p. 168]{N}, we get that $\mathrm{L}_R(M/I)=\sum_{i=1}^ 2\mathrm{L}_{S_{M_i}}[(M/I)_{M_i}][S/M_i:R/M]=\sum_{i=1}^2\mathrm{L}_ {S_{M_i}}[(M/I)_{M_i}]\ (**)$ because $R\subset S$ is infra-integral. Moreover, $(M/I)_{M_1}=(M_1\cap M_2/J\cap M_2)_{M_1}=(M_1)_{M_1}/J _{M_1}$, which implies $\mathrm{L}_{S_{M_1}}[(M/I)_{M_1}]=1$ because $J\in\mathrm{Q_rMax}(S)$ with $J\prec M_1$. In the same way, we have $(M/I)_{M_2}=(M_2)_{M_2}/((M_ 2)_{M_2}\cap J_{M_2})=(M_2)_{M_2}/(M_ 2)_{M_2}=0$ because $J_{M_2}=S_{M_2}$ since $J$ is $M_1$-primary. This implies that $\mathrm{L}_{S_{M_2}}[(M/I)_{M_2}]=0$, so that $\mathrm{L}_R(M/I)=1$, giving that $I\prec M$ in $R$. Then necessarily $I\in\mathrm{Q_rMax}(R)$ by Proposition \ref{2.71} since $I$ cannot be contained in another maximal ideal of $R$ by $(*)$.

At last, assume that $\sqrt[S]J=N'$ for some $N'\in\mathrm{Max}(S),\ N' \neq M_i$ for $i\in\mathbb N_2$. Hence $I\subseteq N:=N'\cap R\neq M$. In particular, $I\not\subseteq M$ because $J_M=S_M$ gives $I_M=R_M$. Then $N\not\in\mathrm{Supp}(S/R)$, so that $I\in\mathrm{Q_rMax}(R)$ and $J=IS$ by Proposition \ref{4.11}. 
\end{proof}

\subsection{Quasi-maximal ideals in minimal ramified   extensions}

Let $R\subset S$ be in $\mathcal M_r$ and $I\in\mathrm{QMax}(R)$ be such that $I\subset (R:S)$. If there exists an ideal $J$ of $S$ such that $I= J\cap R$, according to Lemma \ref{4.3}, we have $IS=I$ and $I$ is an ideal of $S$. To consider the existence of such an ideal $J$, we will examine the case where $I$ is also an ideal of $S$.

\begin{proposition}\label{4.18} Let $R\subset S$ be in $\mathcal M_r$ with $M:=(R:S)$ and $N:=\sqrt[S]M$. Let $I$ be a submaximal ideal of $R$ 
such that $I=IS$, then $I\not\in\mathrm{QMax}(S)$. 
\end{proposition}

\begin{proof} First, $M\in\mathrm{Max}(R)$ and $N\in\mathrm{Max}(S)$. Moreover, $I\subseteq M$ because $I=IS$. Theorem \ref{4.1} says that $M \in\mathrm{Q_rMax}(S)$. But $N\cap R=M$ is also lying over $M$ and $N\in\mathrm{Q_iMax}(S)$. Assume that $IS=I$, so that $I$ is an ideal of $S$. Any ideal of $S$ contained between $I$ and $M$ is an ideal of $R$, that is either $I$ or $M$. In particular, according to \cite[Theorem 12, p.166]{N}, $\mathrm{L}_R(M/I)=1=\mathrm{L}_S(M/I)$. Since $R\subset S$ is in $\mathcal M_r$, $\mathrm{L}_S(N/M)=1$ by Theorem \ref{4.1} implies that $\mathrm{L}_S (N/I)=2$, so that $I\not\in\mathrm{QMax}(S)$.
\end{proof}

The following Proposition generalizes Proposition \ref{4.18} in case $I\in\mathrm{Q_rMax}(R)$.

\begin{proposition}\label{4.13} Let $R\subset S$ be in $\mathcal M_r$ with $M:=(R:S)$ and $N:=\sqrt[S]M$. Let $I\in\mathrm{Q_rMax}(R)$ be such that $I=IS$, so that $I\subset M$. Set $R':=R/I,\ S':=S/I,\ M':=M/I$ and $N': =N/I $. There exists $ J\in\mathrm{QMax}(S)$ such that $I=J\cap R$ if and only if $S'\cong R'(+)(R'/M')$, and in this case $J\in\mathrm{Q_rMax}(S)$. 
\end{proposition}

\begin{proof} We proved in Proposition \ref{4.18} that $I\not\in\mathrm {QMax}(S)$. We are going to characterize $J\in\mathrm{QMax}(S)$ such that $I=J\cap R$. Since $I\in\mathrm{Q_rMax}(R)$ with $I\subseteq M$, Proposition \ref{2.5} says that $R'$ is a SPIR whose maximal ideal $M'$ is nilpotent of index 2. In particular, there exists $t\in M'\setminus\{0\}$ such that $M'=R't$ with $t^2=0$. Moreover, $R'\subset S'$ is also in $\mathcal M_r$ by \cite[Proposition 4.7]{DPPS} with $(R':S')=M'$. In his paper \cite[Theorem 2.1(g)]{Dram}, Dobbs gives a characterization of extensions $R'\subset S'$ in $\mathcal M_r$, where $R'$ is a SPIR. Applied to our situation where the index of nilpotence of $M'$ is 2, we get that $S'$ is an $R'$-algebra isomorphic to either $R'(+)(R'/M')\ (*)$ or $S'_u:=R'[X]/(tX,X^ 2-tu)\ (**)$, for some $u\in\mathrm{U}(R')$. Moreover, in each case, $S'$ is local since so is $R'$. We claim that there exists $J\in\mathrm{QMax}(S)$ such that $I=J\cap R$ if and only if  $(*)$ holds. 

In case $(*)$, $J':=0(+)(R'/M')\in\mathrm{QMax}(S')$ by \cite[Theorem 6]{AKKT} because $0\in\mathrm{QMax}(R')$ according to Theorem \ref{2.3} and more precisely is in $\mathrm{Q_rMax}(S')$. Let $J$ be an ideal of $S$ with $I\subseteq J$ and such that $J'=J/I$. Applying \cite[Corollary 2]{AKKT}, we get that $J\neq I$ and $J\in\mathrm{QMax}(S)$ with $I\subset J$. Moreover, $J'\cap R'=\{x\in R'\mid(x,0)\in 0(+)(R'/M')\}=0$ implies that $J\cap R=I$, and in this case obviously $ J\in\mathrm{Q_rMax}(S)$. 

Assume that $(**)$ holds for some $u\in\mathrm{U}(R')$. Set $y:=\overline X\in S'_u$, where $\overline X$ is the class of $X$ in $S'_u$. Then $S'_u= R'[y]$, with $ty=y^2-tu=0$, so that $y^2=ut\in M'\setminus\{0\}$, which implies $y\in N'\setminus M'$. Therefore $S'_u=R'+R'y$ and $N'=M'+R'y$ since $M'\in\mathrm{Q_rMax}(S'_u)$ with $N'=\sqrt[S'_u] {M'}$. It follows that $N'=R't+R'y=R'y^2+R'y=y(R'+R'y)=S'_uy$ and $N'$ is a principal ideal of $S'_u$ with $y^3=uty=0$. We know that $R'$ is a Noetherian ring, and so is $S'_u$, which is a local ring whose maximal ideal is principal. Then, $ S'_u$ is a SPIR and its ideals are $N'^3=0,N'^2= M'$ and $N'$ which satisfy $0\subset M'\subset N'$. Then, $M'$ is the only submaximal ideal of $S'_u$. In particular, there does not exist $J'\in\mathrm{QMax}(S'_u)$ such that $J'\cap R'=0$, that is there does not exist $J\in\mathrm{QMax}(S_u)$ of $S_u$ such that $J\cap R= I$.
\end{proof}

Here is a converse of Proposition \ref{4.13}.

\begin{proposition}\label{4.16} Let $R\subset S$ be in $\mathcal M_r$ with $M:=(R:S)$ and let $N:=\sqrt[S]M$. Let $J\in\mathrm{QMax}(S)$ with $J\subseteq N$ and $J\not\subseteq M$. Set $I:=J\cap R$. Then $I=IS\in\mathrm{QMax}(R)$ and is homotypic to $J$.
\end{proposition}

\begin{proof} If $J\in\mathrm{Q_iMax}(S)$, then $J=N$ and $I=M=IS\in\mathrm{Q_iMax}(R)$. 

If $J\in\mathrm{Q_dMax}(S)$, then $I\in\mathrm{Q_dMax}(R)$ by Proposition \ref{3.1} because $R\subset S$ is in $\mathcal M_r$, and then an i-extension by \cite[Theorem 2.2]{SPLIT}. Moreover, $I=IS$ by Proposition \ref{4.3}  because $I\subset M$.

If $J\in\mathrm{Q_rMax}(S)$, it satisfies $N^2\subseteq J\subset N$ since $J\subset N$. Therefore, $M^2\subseteq I\subseteq M$ and $I$ is $M$-primary. Moreover, $M$ and $J$ are both $N$-primary. In particular, since $J\subset N$, we have $I=J\cap R\cap N=J\cap M$, so that $I$ is also an ideal of $S$, giving $I=IS$. As $S$-modules as well as $R$-modules, we have the following isomorphisms $M/I=M/(J\cap M)\cong(M+J)/J=N/J$ because $J\not\subseteq M$ with $J$ and $M$ both submaximal in $N$. By \cite[Theorem 13, p.168]{N}, we deduce that $\mathrm{L}_R(M/I)=\mathrm{L}_R(N/J)=\mathrm{L}_{S_N}[(N/J)_N]=\mathrm{L}_S(N/J)=1$. Then, $I$ is submaximal in $M$ as an ideal of $R$, and is then in $\mathrm{Q_rMax}(R)$, because $M$-primary.  
\end{proof}
 
\begin{corollary}\label{4.19} Let $R\subset S$ be in $\mathcal M_r$ with $ M:=(R:S)$ and let $N:=\sqrt[S]M$. Let $J\in\mathrm{QMax}(S)$ with $J\not\subseteq M$. Then $I:=J\cap R\in\mathrm{QMax}(R)$ and is homotypic to $J$.
\end{corollary}

\begin{proof} Since $R\subset S$ is in $\mathcal M_r$, then an i-extension by \cite[Theorem 2.2]{SPLIT}, and according to Proposition \ref{4.16}, if $J\subseteq N$ but $J\not\subseteq M$, then $I\in\mathrm{QMax}(R)$ is homotypic to $J$.

Assume that $J\not \subseteq N$. 

By Proposition \ref{3.1}, if $J\in\mathrm{Q_iMax}(S)\cup\mathrm{Q_dMax}(S)$, then $I\in\mathrm{QMax}(R)$ and is homotypic to $J$, because $R\subset S$ is in $\mathcal M_r$, then an i-extension by \cite[Theorem 2.2]{SPLIT}. 

If $J\in\mathrm{Q_rMax}(S)$ and $J\not\subseteq N$, let $Q:=\sqrt[S]J\neq N$. Setting $P:=Q\cap R$, we have $P\neq M$ because $Q\neq N$ and $I=J\cap R\not\subseteq M$. According to Proposition \ref{4.11}, $I\in\mathrm{Q_rMax}(R)$. 
\end{proof}

Let $(R,M)\subset (S,N)$ be an extension where $(R,M)$ and $(S,N)$ are local rings such that $N$ is lying over $M$. Then, $MS\subseteq N$ is an ideal of $S$. It may or not be a maximal ideal of $S$, as we can see when $M=(R:S)$ for $R\subset S$ in $\mathcal M$ in Theorem \ref{minimal+}. In the next example, we get that some ideal may be quasi-maximal or not according to the  considered ring.

\begin{example}\label{4.22} Let $k$ be a field and set $T:=k[X,Y]/(X^2,Y^ 2,XY):=k[x,y],\ R:=k[x]$ and $S:=k[y]$, where $x$ and $y$ are the classes of $X$ and $Y$ in $T$. This gives the following commutative diagram:
$\begin{matrix}
{} &       {}       & R  &       {}      & {}  \\
{} & \nearrow & {}  & \searrow & {}   \\
k &      {}        & {}  &       {}       & T   \\
{} & \searrow & {}  & \nearrow & {}   \\ 
{} &     {}        & S  &       {}       & {}
\end{matrix}$
where $(R,M),\ (S,N)$ and $(T,P)$ are local rings with $M=kx,\ N=ky$ and $P=kx+ky$. Obviously, $k\subset R$ and $k\subset S$ are in $\mathcal M _r$ with $(k:R)=(k:S)=0$. Moreover, $T=R[y]=S[x]$ satisfies $My=Nx=0$ and $x^2=y^2=0$, so that $R\subset T$ (resp. $S\subset T$) is in $\mathcal M_r$ with $M=(R:T)$ (resp. $N= (S:T)$). Then, Theorem \ref{4.1} gives that $0\in\mathrm{Q_rMax}(S)\cap\mathrm{Q_rMax}(R)$ (resp. $M,\ N\in\mathrm{Q_rMax}(T)$). As $0\subset M\subset P$, we get that $0\not\in\mathrm{QMax}(T)$.
\end{example}

Now, we show how we can build minimal extensions and quasi-maximal ideals from an idealization.  

\begin{proposition}\label{4.21} Let $I\prec J$ be two ideals of a ring $R$. 
Then, $R(+)I\subset R(+)J$ is in $\mathcal M_r$ with $(R(+)I:R(+)J)=N(+)I$, where $N:=(I:J)\in\mathrm{Max}(R)$. In particular, $N(+)I\in\mathrm {Q_rMax}( R(+)J)$. 
\end{proposition}

\begin{proof} Since $J\succ I$, there is no ideal of $R$, and then no $R$-module strictly contained between $I$ and $J$. Therefore $\mathrm{L}_R (J/I)=1$. According to \cite[Proposition 2.8]{Pic 8}, this implies that $\ell[R (+)I,R(+)J]=\mathrm{L_R}(J/I)=1$, so that $[R(+)I,R(+)J]=\{R(+)I,R(+)J\}$ and $R(+)I\subset R(+)J$ is in $\mathcal M$. Moreover, $R\subset R(+)J$ is subintegral by \cite[Lemma 2.1]{Pic 8}, and so is $R(+)I\subset R(+)J$. Then, \cite[Lemma 3.1]{Pic 4} shows that $R(+)I\subset R(+)J$ is in $\mathcal M_r$. In particular, $(R(+)I:R(+)J)\in\mathrm{Max}(R(+)\\I)$, and is of the form $M(+)I$, where $M\in\mathrm{Max}(R)$. Because there exists some $N\in\mathrm{Max}(R)$ such that $NJ\subseteq I$ by \cite[Corollary 1, p.237]{ZS}, this gives that $(R(+)J)(N(+)I)\subseteq N(+)(NJ+I)=N(+)I\subseteq R(+)I$, so that $(R(+)I:R(+)J)=N(+)I$. In particular, $N(+)I\in\mathrm{Q_rMax}(R(+)J)$ by Theorem \ref{4.1}. 
\end{proof}

\subsection{Quasi-maximal ideals in minimal flat epimorphisms}

Although the situation is quite different for a minimal flat epimorphism 
 $R\subset S$, we can still consider the conductor $(R:S)$. Before, we first give a characterization of such extensions. 

\begin{proposition}\label{4.31} Let $R\subset S$ be an extension. The following conditions are equivalent:
\begin{enumerate}
\item $R\subset S$ is a minimal flat epimorphism.
\item $\mathrm{Supp}_R(S/R)=\{M\}$ for some $M\in\mathrm{Max}(R)$ and $R_M\subset S_M$ is a  minimal flat epimorphism.
\item $\mathrm{Supp}_R(S/R)=\{M\}$ for some $M\in\mathrm{Max}(R)$ and $(R:S)=:P\in\mathrm{Spec}(R)$ is such that $P\subset M$ with $S_M\cong R_P$ and $R_M/PR_M$ is a one-dimensional valuation domain.
\end{enumerate} 
\end{proposition}
\begin{proof} (1)$\Leftrightarrow$(2) by \cite[Proposition 4.6]{DPPS} and \cite[Th\'eor\`eme 2.2]{FO}.

(2)$\Leftrightarrow$(3) by \cite[Proposition 9.2 and Corollary 9.3]{ARP}.
\end{proof}

The first result of Theorem \ref{4.1} also holds for minimal flat epimorphism
 $R\subset S$ under additional condition.

\begin{proposition}\label{4.32} Let $R\subset S$ be a minimal flat epimorphism. Then, $(R:S)\in\mathrm{QMax}(S)$ if and only if $|\mathrm {V}_R ((R:S))\cap\mathrm{Max}(R)|=1$. In this case, $(R:S)\in\mathrm {Max} (S)$.
\end{proposition}
\begin{proof} Since $R\subset S$ is a minimal flat epimorphism, $P:=(R:S)\in\mathrm{Spec}(R)\cap\mathrm{Spec}(S)$ by \cite[Lemme 3.2]{FO}, is such that $P\subset M$, where $\{M\}=\mathrm{Supp}_R(S/R)$, and there is no prime ideal of $R$ contained strictly between $P$ and $M$ according to \cite[Corollary 9.3]{ARP}. Moreover, the spectral map $\mathrm{Spec}(S)\to\mathrm{Spec}(R)\setminus\{M\}$ is a bijection by \cite[Theorem 2.1]{DPP2}. 

Assume that $|\mathrm{V}_R(P)\cap\mathrm{Max}(R)|=1$, which means that $M$ is the only maximal ideal of $R$ containing $P$. We claim that $ P\in\mathrm{Max}(S)$. Otherwise, there is $Q\in\mathrm{Max}(S)$ such that $P\subset Q$, so that $P\subseteq Q\cap R\in\mathrm{Max}(R)$, a contradiction because $P,M\neq Q\cap R$. Then $P\in\mathrm{QMax}(S)$.

Assume now that $|\mathrm{V}_R(P)\cap\mathrm{Max}(R)|>1$ and let $N\in\mathrm{Max}(R),\ N\neq M$ be such that $P\subset N$. Since $N\not\in\mathrm{Supp}_R(S/R)$, it follows that $NS\in\mathrm{Max}(S)$ (see Lemma \ref{3.46}). Then, Proposition \ref{2.7}, Theorem \ref{2.6} and $P\subset NS$ in $S$ show that $P\not\in \mathrm{QMax}(S)$. 
\end{proof}

\begin{example}\label{4.33} This example shows that the conductor of a minimal flat epimorphism $R\subset S$ is not always quasi-maximal in $S$.

Set $R:=\mathbb Z$ and let $p$ be a prime number. Consider $S:=\mathbb Z_p$ and $M:=p\mathbb Z$. Then $M$ is the only maximal ideal of $R$ with no prime ideal of $S$ lying over it. Obviously, $R\subset S$ is a minimal flat epimorphism such that $MS=S,\ \{M\}=\mathrm{Supp}_R(S/R)$ and $(R:S)=0$, which is not in $\mathrm{QMax}(S)$. In fact, any maximal ideal of $R$ contains $(R:S)$. We may remark that $R_M=\mathbb Z_{p\mathbb Z}$ is a discrete valuation domain with quotient field $S_M=\mathbb Q$ (see \cite[Proposition 4.6]{DPPS}). In particular, this example shows that Proposition \ref{3.16} has no converse when $R_M$ is a local ring since $0$ is a maximal ideal in $S_M$ and a prime not maximal ideal in $R_M$, and then not quasi-maximal in $R_M$.
\end{example}

\section{Quasi-maximal ideals and 2-absorbing ideals}

Badawi introduced in \cite{Ba} the notion of 2-absorbing ideals: a nonzero proper ideal $I$ of a ring $R$ is said to be 2-absorbing if for any $a,b,c\in R$ such that $abc\in I$, then $ab\in I$ or $bc\in I$ or $ac\in I$. As a 
quasi-maximal ideal is a 2-absorbing ideal by \cite[Theorem 1]{AKKT}, and they are quite similar by the trichotomy of their characterization, we look in this section at situations where these two notions coincide. In \cite[Theorem 2.4]{Ba}, Badawi shows that a 2-absorbing ideal $I$ of a ring $R$ satisfies one of the next properties:

(a) $\sqrt I$ is a prime ideal $P$ of $R$ such that $P^2\subseteq I$.

(b) $\sqrt I=P_1\cap P_2,\ P_1P_2\subseteq I$ and $(\sqrt I)^2\subseteq I$, where $P_1,P_2$ are the only distinct prime ideals of $R$ that are minimal over $I$.

(A prime ideal minimal over $I$ is a minimal ideal in 
 $\mathrm{V}_R(I)$).
 
\begin{remark}\label{5.1} If we limit for the characterization of 2-absorbing ideals which are maximal ideals, Proposition \ref{2.70} gives for (a): Either $I\in\mathrm{Max}(R)$, giving Proposition \ref{2.70}(1) or $I$ is an $M$-primary ideal, for some $M\in\mathrm{Max}(R)$ such that $M^2\subseteq I\subset M$ giving Proposition \ref{2.70}(3). And (b) gives $I=M_1\cap M_2 $ for two distinct maximal ideals $M_1,M_2$ of $R$ since $M_1M_2 \subseteq I\subset M_1\cap M_2$ gives $I=M_1\cap M_2$ because $M_1$ and $M_2$ are comaximal, which is Proposition \ref{2.70}(2).
\end{remark}
 
We get the following converse of \cite[Theorem 1]{AKKT}.

\begin{theorem}\label{5.2} An ideal $I$ of a ring $R$ is in $\mathrm{QMax}(R)$ if and only if $I$ is a 2-absorbing ideal such that $\mathrm{V}(I)\subseteq\mathrm{Max}(R)$ and $M/I$ is a principal ideal of $R/I$ for each $M\in \mathrm{V}(I)$. 
\end{theorem}

\begin{proof} One implication is \cite[Theorem 1 and Definition 1]{AKKT} and Theorem \ref{2.3}. 

Conversely, assume that $I$ is a 2-absorbing ideal such that $\mathrm{V}(I)\subseteq\mathrm{Max}(R)$ and $M/I$ is a principal ideal of $R/I$ for each $M\in\mathrm{V}(I)$. First, $I$ is a proper ideal by definition of a 2-absorbing ideal. According to \cite[Theorem 2.3]{Ba} and Remark \ref{5.1}, $I\in\mathrm{QMax}(R)$ if either $I$ is a maximal ideal (first part of case (a) of Remark \ref{5.1}) or an intersection of two maximal ideals (case (b)  of Remark \ref{5.1}) by Theorem \ref{2.3}. In case (a) of Remark \ref{5.1} when $\sqrt I$ is a maximal ideal $M$ of $R$ such that $M^2\subseteq I$ with $I\subset M$, we get that $I$ is an $M$-primary ideal such that $M^2 \subseteq I\subset M$. Moreover, $M/I$ is a principal ideal by the assumption, so that there exists $t\in M$ such that $M=I+Rt$. As $M^2 \subseteq I$, this implies that $M\succ I$ by \cite[Corollary 2, p.237]{ZS}. An application of Theorem \ref{2.6} shows that $I\in\mathrm{QMax}(R)$.
\end{proof} 

\begin{corollary}\label{5.3} Let $I\in\mathrm{Q_rMax}( R)$ and $M:=\sqrt I$. Then $(I:x)=M$ for any $x\in M\setminus I$.
\end{corollary}

\begin{proof} Since $I$ is  2-absorbing by Theorem \ref{5.2}, then $(I:x)=M$ for any $x\in M\setminus I$ by \cite[Theorem 2.5]{Ba}. 
\end{proof} 

\begin{corollary}\label{5.4} Let $M$ be a maximal ideal of a ring $R$ such that $M^2\neq M$. Then $M^2\in\mathrm{QMax}(R)$ if and only if $M/M^2$ is a principal ideal if and only if $\mathrm{L}_R(M/M^2)=1$.
\end{corollary}

\begin{proof} By Theorem \ref{5.2}, $M^2\in\mathrm{QMax}(R)$ if and only if $M^2$ is a 2-absorbing ideal such that $\mathrm{V}(M^2)\subseteq\mathrm{Max}(R)$ and $M/M^2$ is a principal ideal. By \cite[Theorem 3.1]{Ba}, $M^2$ is a 2-absorbing ideal. Obviously, $\mathrm{V}(M^2)=\{M\}\subseteq\mathrm{Max}(R)$. Then, the first equivalence is gotten. At last, $M^2\in\mathrm{QMax}(R)$ if and only if $\mathrm{L}_R(M/M^2)=1$ by  Theorem \ref{2.3} since $M^2\in\mathrm{QMax}(R)$ if and only if $M^2\in\mathrm{Q_rMax}( R)$.
\end{proof} 

\begin{remark}\label{5.14} 
Theorem \ref{5.2} shows that in a ring $R$, the class of 2-absorbing ideals is larger than the class of quasi-maximal ideals, even if we limit to ideals $I$ such that $\mathrm{V}(I)\subseteq\mathrm{Max}(R)$. In particular, for an extension $R\subset S$, if $J$ is a 2-absorbing ideal of $S$, then $J\cap R$ is obviously a 2-absorbing ideal of $R$.  But, in Remark \ref{4.9} we give an example of an extension $R\subset S$ in $\mathcal M_i$ and an ideal $J$ of $S$ which is also an ideal of $R$ and such that $J\in\mathrm{Q_rMax}(S)$ with $J\not\in\mathrm {QMax}(R)$.
\end{remark}

We are now looking how some properties gotten by Badawi for 2-absorbing ideals in \cite{Ba} are transferred  for quasi-maximal ideals.

\begin{proposition}\label{5.8} An $M$-primary ideal $I$ of a valuation domain $(R,M)$ is 2-absorbing if and only if $I\in\mathrm{QMax}(R)$.
\end{proposition}

\begin{proof} If $I\in\mathrm{QMax}(R)$, then $I$ is 2-absorbing by \cite[Theorem 1]{AKKT}. 

Conversely, assume that $I$ is 2-absorbing. Since $I$ is an $M$-primary ideal and according to \cite[Proposition 3.10]{Ba}, either $I=M$ or $I=M^2 $. In the first case, $I\in\mathrm{QMax}(R)$ by Definition \ref{2.2}. Assume that $M\neq I=M^2$. Then $I\subset M$ is submaximal because the only $ M$-primary ideals of $R$ are powers of $M$ by \cite[Theorem 5.11]{LM}. Therefore Theorem \ref{2.6} gives that $I\in\mathrm{QMax}(R)$.
\end{proof} 

\begin{proposition}\label{5.9} A nonzero ideal $I$ of a Pr\"ufer domain $R$ such that $\mathrm{V}(I)\subseteq\mathrm{Max}(R)$
 is 2-absorbing if and only if $I\in\mathrm{QMax}(R)$.
\end{proposition}

\begin{proof} If $I\in\mathrm{QMax}(R)$, then $I$ is 2-absorbing by \cite[Theorem 1]{AKKT}. 

Conversely, assume that $I$ is 2-absorbing. According to \cite[Theorem 3.14]{Ba}, either $I=M\ (*)$ or $I=M^2\ (**)$ or $I=M_1\cap M_2\ (***)$,  where $M,M_1,M_2\in\mathrm{Max}(R)$ since $\mathrm{V}(I)\subseteq\mathrm{Max}(R)$. It follows that $|\mathrm{V}(I)|\leq 2$. By \cite[Lemma 3.13]{Ba}, $IR_M$ is 2-absorbing in $R_M$, and then is in $\mathrm {QMax}(R_M)$ for any $M\in\mathrm{V}(I)$ by Proposition \ref{5.8} because $R_M$ is a valuation ring. Since $IR_M\in\mathrm{Max}(R)$ when $|\mathrm{V}(I)|=2$ by $(***)$,  Proposition \ref{3.161} gives that $I\in\mathrm{QMax}(R)$.
\end{proof} 

\begin{proposition}\label{5.5} A Noetherian domain $R$ which is not a field is a Dedekind domain if and only if $M^2\in\mathrm{QMax}(R)$ for each $M\in\mathrm{Max}(R)$.
\end{proposition}

\begin{proof} In \cite[Theorem 38.1]{MIT},  Gilmer shows that a Noetherian  domain $R$ which is not a field is a Dedekind domain if and only if, for each $M\in\mathrm{Max}(R)$, there are no ideals properly between $M$ and $M^2$, which is equivalent, if $M\neq M^2$, to $\mathrm{L}_R(M/M^2)=1$ and to $M^2\in\mathrm{QMax}(R)$ by Corollary \ref{5.4}. If $M=M^2$, then $M^2$ is obviously quasi-maximal. 
\end{proof} 

\begin{corollary}\label{5.6} (1) An ideal $I$ of a Dedekind domain $R$ 
 is 2-absorbing if and only if $I\in\mathrm{QMax}(R)$.
 
(2) A Noetherian domain $R$ which is not a field is a Dedekind domain if and only if any 2-absorbing ideal is in $\mathrm {QMax}(R)$.
\end{corollary}

\begin{proof} (1) If $R$ is a field there is nothing to prove. So, assume that $R$ is not a field. If $I\in\mathrm{QMax}(R)$, then $I$ is a 2-absorbing ideal by \cite[Theorem 1]{AKKT}. If $I$ is 2-absorbing, then \cite[Theorem 3.15]{Ba} gives that there exists some $M\in\mathrm{Max}(R)$ such that either $I=M$ or $I=M_1\cap M_2$ with $M=M_i$ for $i\in\{1,2\}$ or $I=M^2 $ which is in $\mathrm{QMax}(R)$ by Proposition \ref{5.5}. In the two first cases, $I$ is obviously also in $\mathrm{QMax}(R)$ by Definition \ref{2.2}.

(2) If $R$ is a Dedekind domain, then any 2-absorbing ideal is in $\mathrm{QMax}(R)$ by (1).

Conversely, assume that any 2-absorbing ideal is in $\in\mathrm{QMax}(R)$. By \cite[Theorem 3.1]{Ba}, $M^2$ is a 2-absorbing ideal for any $M\in\mathrm{Max}(R)$. Then Proposition \ref{5.5} gives that $R$ is a Dedekind domain. 
\end{proof} 

We end this section by adding some results perhaps already known about the behavior of 2-absorbing ideals through ring morphisms. 

\begin{proposition}\label{5.11} Let $R\subseteq S$ be a flat epimorphism, $J$ a proper ideal of $S$ and $I:=J\cap R$. Then, 
\begin{enumerate}
\item If $J$ is a 2-absorbing ideal of $S$, then $I$ is a 2-absorbing ideal of $R$. 
\item Assume that $I$ is an $M$-primary ideal with $M\in\mathrm{Max}(R)$. Then, $J$ is a 2-absorbing  if and only if $I$ is a 2-absorbing ideal of $R$.
\end{enumerate}
\end{proposition}
\begin{proof} (1) Obvious.

(2) One implication is (1). Now, assume that $I$ is an $M$-primary 2-absorbing ideal with $M\in\mathrm{Max}(R)$. By \cite[Theorem 3.1]{Ba}, $ M^2\subseteq I\ (*)$. According to \cite[Scholium A (6)]{Pic 5}, $J=IS$. Since $J$ is a proper ideal of $S$, then $J\subseteq N$ for some   $N\in \mathrm{Max}(S)$. Moreover, $M^2\subseteq I=J\cap R\subseteq N\cap R$ implies that $M=N\cap R$, so that $N=MS$. Then, we get by $(*)$ that $ N^2\subseteq J=IS$. Hence, $J$ is an $N$-primary ideal of $S$. The previous reference shows that  $J$ is a 2-absorbing ideal of $S$.
\end{proof} 

\begin{corollary}\label{5.12} (1) An $M$-primary ideal $I$ of a ring $R$ with $M\in\mathrm{Max}(R)$ is 2-absorbing if and only if $IR_M$ is a 2-absorbing ideal of $R_M$.

(2) An $M$-primary ideal $I$ of a ring $R$ with $M\in\mathrm{Max}(R)$ is  2-absorbing if and only if $IR(X)$ is a 2-absorbing ideal of $R(X)$.
\end{corollary}
\begin{proof} (1) One implication is \cite[Lemma 3.13]{Ba}. For the converse, we use Proposition \ref{5.11}(2) with $S:=R_M$ and $J:=IR_ M$.

(2) If $I$ is a 2-absorbing ideal of $R$, then $M^2\subseteq I\subseteq M$ by \cite[Theorem 3.1]{Ba}. Therefore $(MR(X))^2\subseteq IR(X)\subseteq MR(X) $ which gives first that $IR(X)$ is an $MR(X)$-primary ideal of $R(X)$ and then a 2-absorbing ideal of $R(X)$ by the previous reference.

If $IR(X)$ is a 2-absorbing ideal of $R(X)$ with $I$ an $M$-primary ideal of $R$, then $IR(X)$ is an $MR(X)$-primary ideal of $R(X)$ by \cite[Corollary 1, p.153]{ZS}, because $MR(X)\in\mathrm{Max}(R(X))$ and $M^2R(X)=(M R(X))^2\subseteq IR(X)$ implies $M^2\subseteq I$, so that $I$ is a 2-absorbing ideal of $R$, always by \cite[Theorem 3.1]{Ba}.
\end{proof} 

\section{Appendix}

For a proper ideal $I$ of a ring $R$, the first author defines in \cite[D\'efinitions 1, 2 and 3]{PICANAL} the set $\Lambda (I):=\{x\in R\mid I:x=I\}$, which is a saturated multiplicative set. It satisfies $I\cap\Lambda(I)=\emptyset$ and $\Lambda(I)$ is the set of elements of $R$ whose image in $R/I$ is regular. Then, $\mathrm{t}(R/I)=(R/I)_{\Lambda (I)}$ follows.

\begin{proposition}\label{2.82} Let $R$ be a ring and $I\in\mathrm{QMax}(R)$.  We have the following properties: 
\begin{enumerate}
\item If $I\in\mathrm{Q_iMax}(R)\cup\mathrm{Q_rMax}(R)$ and $M\in\mathrm{Max}(R)\cap\mathrm{V}(I)$, then $\Lambda(I)=R\setminus M$.
\item If $I\in\mathrm{Q_dMax}(R)$ with $I=M_1\cap M_2,\ M_i\in\mathrm {Max}(R)$ for $i\in\mathbb N_2$, then $\Lambda(I)=R\setminus(M_1\cup M_2)=(R\setminus M_1)\cap(R\setminus M_2)$.
\end{enumerate} 
\end{proposition}
\begin{proof} $\Lambda(I)$ is the set of elements of $R$ whose image in $R/I$ is regular.

(1) If $I\in\mathrm{Q_iMax}(R)$, then $I=M$ and $R/I$ is a field, so that $\Lambda (I)=R\setminus M$. 

If $I\in\mathrm{Q_rMax}(R)$, let $M:=\sqrt I\in\mathrm{Max}(R)$. Proposition \ref{2.5} says that $R/I$ is a SPIR whose maximal ideal $M/I$ is nilpotent, so that $\Lambda (I)=R\setminus M$. 

(2) Let $I\in\mathrm{Q_dMax}(R)$, with $I=M_1\cap M_2,\ M_i\in\mathrm {Max}(R)$ for $i\in\mathbb N_2$. Then $R/I\cong R/M_1\times R/M_2$, a product of two fields. Hence, its set of regular elements is $(R\setminus M_ 1)/M_1\times(R\setminus M_2)/M_ 2$, so that $\Lambda(I)=(R\setminus M_1)\cap(R\setminus M_2)=R\setminus(M_1\cup M_2)$.
\end{proof}

\begin{corollary}\label{2.83} Let $R$ be a ring and $I$ an ideal of $R$ such that $I\prec M$ for some $M\in\mathrm{Max}(R)$. Then $I\in\mathrm {Q_rMax}( R)\cup\mathrm{Q_dMax}( R)$. Let $x\not\in M$. Then:
\begin{enumerate}
\item If $I\in\mathrm{Q_rMax}(R)$, then $I:x=I$ and $x\in\Lambda (I)$.
\item If $I\in\mathrm{Q_dMax}(R)$ with $I=M\cap M',\ M'\in\mathrm{Max}(R)$ and $M\neq M'$, then:
\begin{enumerate}
\item $I:x=M$ and $x\not\in \Lambda(I)$ if  $x\in M'\setminus M$,
\item $I:x=I$ and $x\in\Lambda(I)$  if  $x\not\in M\cap  M'$.\end{enumerate} 
\end{enumerate}
\end{corollary}
\begin{proof} By Theorem \ref{2.6}, $I\in\mathrm{Q_rMax}(R)\cup\mathrm{Q_dMax}( R)$. 

Let $x\not\in M$ and $a\in I:x$. Since $ax\in I\subseteq M\ (*)$, we get $a\in M$, so that $I\subseteq I:x\subseteq M$. Therefore $I:x\in\{I,M\}$ because $I\prec M$.

(1) If $I\in\mathrm{Q_rMax}(R)$, then $I$ is $M$-primary, and $(*)$ implies that $a\in I$, so that $I:x\subseteq I\subseteq I:x$, giving $I= I:x$ and $x\in\Lambda (I)$. 

(2) Assume that $I\in\mathrm{Q_dMax}(R)$, with $I=M\cap M',\ M'\in\mathrm {Max}(R),$

\noindent $ M\neq M'$. 

(a) If $x\in M'\setminus M$, then any $a\in M\setminus I$ satisfies $ax\in I$, which implies $a\in(I:x)\setminus I$. The beginning of the proof shows that $I:x=M$ and $x\not\in \Lambda(I)$.

(b) If $x\not\in M\cup M'$, then $x\in R\setminus(M\cup M')=\Lambda(I)$ by Proposition \ref{2.82} and $I= I:x$.
\end{proof}

\begin{proposition}\label{2.84} Let $f:R\to S$ be a ring morphism and let $ J$ be an ideal of $S$. Then $f^{-1}[\Lambda(J)]\subseteq\Lambda[f^{-1} (J)]$, with equality if $f$ is surjective.
\end{proposition}
\begin{proof} According to \cite[Corollaire, p.89]{PICANAL}, $f^{-1}[\Lambda(J)]\subseteq \Lambda[f^{-1}(J)]$ holds.

We claim that $f^{-1}[\Lambda(J)]=\Lambda[f^{-1}(J)]$ if $f$ is surjective. It is enough to show that $\Lambda[f^{-1}(J)]\subseteq f^{-1}[\Lambda(J)]$. So, set $I:=f^{-1}(J)$ and let $x\in\Lambda(I)$, that is $I:x=I$. Setting $y:=f (x)\in S$, we already have $J\subseteq J:y$. Let $b\in J:y\subseteq S$, so that $by\in J\ (*)$. Since $f$ is surjective, there exists some $a\in R$ such that $b =f(a)$. By $(*)$, we have $f(ax)\in J$. As $ax\in I$, we infer that $a\in I:x=I$. Hence, $b=f(a)\in f(I)\subseteq J$. We proved that $J:y\subseteq J$, which means that $J:y=J$ and $y=f(x)\in\Lambda(J)$. To conclude, $x\in f^{-1}[\Lambda(J)]$, and our claim is gotten. 
\end{proof}

\begin{proposition}\label{2.85} Let $f:R\to S$ be a ring morphism and let $  J$ be an ideal of $S$. If $I:=f^{-1}(J)$, then $f^{-1}(J:f(x))=I$ for any $x\in\Lambda(I)$.
\end{proposition}
\begin{proof} Let $x\in\Lambda(I)$, so that $I:x=I$. Set $y:=f(x)$ and $K:= J:y$. Now, let $z\in f^{-1}(K)$. Therefore  $f(z)\in K$ and $f(zx)\in J$. Then, $zx\in I$, whence $z\in I:x=I$. It follows that $f^{-1}(K)=f^{-1}(J:f(x))\subseteq I=f^{-1}(J)\subseteq f^{-1}(J:f(x))$ because $J\subseteq J:f(x)$ always holds. Hence, $f^{-1}(J:f(x))=I$.
\end{proof}

We recall that a proper ideal $I$ of a ring $R$ is called {\it primal} if $\mathrm Z(R/I)=P/I$, where $P\in\mathrm{Spec}(R)$ \cite[Ch. IV, Exercice 33, p.177]{ALCO1}. Then, by \cite[Corollaire, p.83]{PICANAL}, a proper ideal $I$ of a ring $R$ is primal if and only if $R\setminus \Lambda(I)\in\mathrm{Spec}(R)$. This implies the following.

\begin{proposition}\label{2.87} An ideal either inert or ramified is primal.\end{proposition}
\begin{proof} Let $I$ be a proper ideal of a ring $R$. By Proposition \ref{2.82}, if $I\in\mathrm{Q_rMax}(R)\cup\mathrm{Q_iMax}(R)$ with $M:=\sqrt I$, then $\Lambda(I)=R\setminus M$, so that $R\setminus\Lambda(I) =M\in\mathrm{Max}(R)$ and $I$ is primal by the previous characterization. 
\end{proof}

\begin{remark}\label{2.88}(1) We define a proper ideal $I$ of a ring $ R$ as {\it semi-primal} if $R\setminus\Lambda(I)$ is a finite union of prime ideals. Then, a decomposed ideal is semi-primal by Proposition \ref{2.82}.

(2) An ideal $I$ of a ring $R$ is  primal (resp. semi-primal)  if and only if $\mathrm{t}(R/I)$ is local (resp. semi-local). Since $\mathrm{t}(R/I)$ is the localization of $R/I$ at the set of regular elements of $R/I$, that is $(R/I)\setminus \mathrm Z(R/I)$, we get the equivalence.

(3) A primal ideal $I$ of a ring $R$ does not need to be quasi-maximal, even if $R\setminus\Lambda(I)=M\in\mathrm{Max}(R)$. Take a local ring $(R,M)$ such that $M^3=0$, with $M^2\neq M,0$. Theorem \ref{2.3} shows that $0$ is not quasi-maximal, but is primal because $\mathrm Z(R)=M$.
\end{remark}

\begin{proposition}\label{2.86} Let $f:R\to S$ be a ring morphism, $J$ be an ideal of $S$ and $I:=f^{-1}(J)$. Then, there is a Max-upper $K$ in $S$ with $J\subseteq K$, such that $f^{-1}[\Lambda(K)]=\Lambda[f^{-1}(K)]$. Therefore there is a unique ring morphism $g:\mathrm{t}(R/I)\to\mathrm{t}(S/K)$ induced by $f$.
\end{proposition}
\begin{proof} By Lemma \ref{3.53}, there exists an ideal $K$ of $S$ which is a Max-upper of $I$. Now, by Proposition \ref{2.85}, $f^{-1}(K:f(x))=I$ for any $x\in\Lambda(I)$. As $K\subseteq K:f(x)$, we get that $ K=K:f(x)$ by maximality of $K$. Then, $f(x)\in\Lambda(K)$ and $x\in f^{-1}(\Lambda(K))$. It follows that $\Lambda[f^{-1}(K)]\subseteq f^{-1}[\Lambda(K)]\subseteq\Lambda[f^{-1}(K)]$ by Proposition \ref{2.84} and $f^{-1}[\Lambda(K)]=\Lambda[f^{-1}(K)]$ follows. 

Since $I=f^{-1}(K)$, there exists a unique ring morphism $g:\mathrm{t}(R/I)\to\mathrm{t}(S/K)$ induced by $f$.
\end{proof}

\begin{remark}\label{6.1} The use of 2-absorbing ideals recovers in another way the result gotten in Proposition \ref{2.87} saying that an ideal either inert or ramified is primal. Now \cite[Corollary 2.7]{Ba} shows that a 2-absorbing ideal $I$ such that $I\neq\sqrt I$ is primal, excluding the decomposed case. 
\end{remark}

\end{document}